\documentclass[12pt]{amsart}
\usepackage[utf8]{inputenc}
\usepackage{amsmath}
\usepackage{amssymb}
\usepackage{mathrsfs}
\usepackage{graphicx}

\usepackage[colorlinks]{hyperref}
\usepackage[svgnames, dvipsnames, usenames]{xcolor}
\hypersetup{urlcolor = RedViolet, linkcolor = RoyalBlue, citecolor = ForestGreen}

\usepackage{tikz}
\usetikzlibrary{matrix}
\usetikzlibrary{arrows}

\usepackage{partmac}

\theoremstyle{plain} 
\newtheorem{thm}{Theorem}[section]
\newtheorem{prop}[thm]{Proposition}
\newtheorem{lem}[thm]{Lemma}

\theoremstyle{definition}
\newtheorem{defn}[thm]{Definition}
\newtheorem{rem}[thm]{Remark}
\newtheorem{ex}[thm]{Example}

\numberwithin{equation}{section}

\theoremstyle{plain}
\newtheorem{innercustomthm}{Theorem}
\newenvironment{customthm}[1]
  {\renewcommand\theinnercustomthm{#1}\innercustomthm}
  {\endinnercustomthm}

\renewcommand{\theta}{\vartheta}
\renewcommand{\phi}{\varphi}
\renewcommand{\epsilon}{\varepsilon}
\renewcommand{\subset}{\subseteq}
\renewcommand{\supset}{\supseteq}

\newcommand{\ttimes}{\mathbin{\ooalign{$\times$\cr\kern.2ex\raise.2ex\hbox{$\times$}}}}
\newcommand{\timess}{\mathbin{\ooalign{\kern.2ex$\times$\cr\raise.2ex\hbox{$\times$}}}}

\newcommand{\N}{\mathbb N}
\newcommand{\Z}{\mathbb Z}

\newcommand{\C}{\mathbb C}

\DeclareMathOperator{\Mor}{Mor}
\DeclareMathOperator{\Rep}{Rep}
\DeclareMathOperator{\Reptil}{\widetilde{Rep}}
\DeclareMathOperator{\spanlin}{span}

\DeclareMathOperator{\Alt}{Alt}
\DeclareMathOperator{\rl}{rl}

\newcommand{\even}{_{\mathrm{even}}}
\newcommand{\twocol}{^{\wcol\bcol}}

\newcommand{\nlin}{_{N\mathchar `\-\mathrm{lin}}}

\newcommand{\nnnlin}{_{(N-1)\mathchar `\-\mathrm{lin}}}
\newcommand{\Lin}{\mathscr{L}}
\newcommand{\Cat}{\mathscr{C}}
\newcommand{\Kat}{\mathscr{K}}

\newcommand{\V}{\mathcal{V}}
\newcommand{\U}{\mathcal{U}}
\newcommand{\PP}{\mathcal{P}}

\newcommand{\Part}{\mathscr{P}}

\newcommand{\Alphabet}{\mathscr{A}}
\newcommand{\Blphabet}{\mathscr{B}}

\newcommand{\tiltimes}{\mathbin{\tilde\times}}
\newcommand{\tilstar}{\mathbin{\tilde *}}

\usepackage{array}
\newcolumntype{P}[1]{>{\centering\arraybackslash}p{#1}}

\newcommand{\lincol}{\kern0.1ex%
	\pgfpicture
	\pgfsetlinewidth{\Pwidth em}
	\pgfpathmoveto{\pgfpointorigin}
	\pgfpathlineto{\pgfpoint{0em}{0.5em}}
	\pgfusepath{stroke}
	\endpgfpicture\kern0.1ex}

\newcommand{\Psa}{%
	\Partition{
	\Ut (1,1)
	\Psingletons0to0.3:1
	}}

\newcommand{\Pabcacb}{\Partition{
	\Pline (3,1) (1,0)
	\Pline (1,1) (2,0)
	\Pline (2,1) (3,0)
}}

\newcommand{\Pabcbac}{\Partition{
	\Pline (3,1) (2,0)
	\Pline (2,1) (1,0)
	\Pline (1,1) (3,0)
}}

\newcommand{\dotcol}{\kern0.1ex
	\pgfpicture
	\pgfsetlinewidth{\Pwidth em}
	\pgfpathmoveto{\pgfpointorigin}
	\pgfpathlineto{\pgfpoint{0em}{0.5em}}
	\pgfsetdash{{\pgflinewidth}{0.1em}}{0.1em}
	\pgfusepath{stroke}
	\endpgfpicture\kern0.1ex}

\def\PDline(#1,#2)#3(#4,#5){
\pgfpathmoveto{\pgfpointxy{#1}{#2}}
\pgfpathlineto{\pgfpointxy{#4}{#5}}
\pgfsetdash{{\pgflinewidth}{1pt}}{0pt}
\pgfusepath{stroke}
\pgfsetdash{}{0pt}
}

\newcommand{\PDabcb}{\Partition{
\Psingletons 0to0.3:1
\Psingletons 1to0.7:2
\PDline(2,0)(1,1)
}}

\newcommand{\PDabcdcb}{\Partition{
\Psingletons 0to0.3:1
\Psingletons 1to0.7:3
\PDline(2,0)(1,1)
\PDline(3,0)(2,1)
}}

\newcommand{\PDabcabc}{\Partition{
\PDline (1,0) (3, 1)
\PDline (2,0) (2, 1)
\PDline (3,0) (1, 1)
}}
\newcommand{\PDabcabd}{\Partition{
\PDline (1,0) (3, 1)
\PDline (2,0) (2, 1)
\Psingletons 0to0.3:3
\Psingletons 1to0.7:1
}}
\newcommand{\PDabcadc}{\Partition{
\PDline (1,0) (3, 1)
\PDline (3,0) (1, 1)
\Psingletons 0to0.2:2
\Psingletons 1to0.8:2
}}
\newcommand{\PDabcdbc}{\Partition{
\PDline (2,0) (2, 1)
\PDline (3,0) (1, 1)
\Psingletons 0to0.3:1
\Psingletons 1to0.7:3
}}

\newcommand{\Pabcdcb}{\Partition{
\Psingletons 0to0.3:1
\Psingletons 1to0.7:3
\Pline(2,0)(1,1)
\Pline(3,0)(2,1)
}}

\newcommand{\Pabcdce}{\Partition{
\Psingletons 0to0.3:1,3
\Psingletons 1to0.7:2,3
\Pline(2,0)(1,1)
}}

\newcommand{\Pabcdeb}{\Partition{
\Psingletons 0to0.3:1,2
\Psingletons 1to0.7:1,3
\Pline(3,0)(2,1)
}}

\begin{document}
\title{New products and $\Z_2$-extensions of compact matrix quantum groups}
\author{Daniel Gromada and Moritz Weber}
\address{Saarland University, Fachbereich Mathematik, Postfach 151150,
66041 Saarbr\"ucken, Germany}
\email{gromada@math.uni-sb.de}
\email{weber@math.uni-sb.de}
\date{\today}
\thanks{Both authors were supported by the collaborative research centre SFB-TRR 195 ``Symbolic Tools in Mathematics and their Application''. The second author was also supported by the DFG project ``Quantenautomorphismen von Graphen''. The article is part of the first author's PhD thesis.}
\thanks{We thank Amaury Freslon for discussions on the topic, reading a draft of the article and providing helpful comments.}

\begin{abstract}
There are two very natural products of compact matrix quantum groups: the tensor product $G\times H$ and the free product $G*H$. We define a number of further products interpolating these two. We focus more in detail to the case where $G$ is an easy quantum group and $H=\hat\Z_2$, the dual of the cyclic group of order two. We study subgroups of $G*\hat\Z_2$ using categories of partitions with extra singletons. Closely related are many examples of non-easy bistochastic quantum groups.
\end{abstract}

\maketitle
\section*{Introduction}

Quantum groups are a generalization of the concept of a group in non-commutative geometry. In this work, we deal with compact quantum groups as defined by Woronowicz in \cite{Wor87}. One of the main motivations of Woronowicz was to describe the quantum deformation of the special unitary groups. Nevertheless, a lot of effort was put recently in study of quantum groups arising not by deforming commutation relations, but rather by removing or liberating them. First examples in this direction were the free orthogonal, free unitary and free symmetric groups defined by Wang \cite{Wan95free,Wan98}. A new approach for the construction of such quantum groups was introduced by Banica and Speicher in \cite{BS09} inventing the so-called {\em easy quantum groups}. We extend their work in the following sense.

A classical problem in many fields of algebra is the one of constructing extensions. In group theory, for example, the simplest construction considering two groups $G$ and $H$ is to construct their direct product $G\times H$. In the theory of quantum groups, we can repeat this construction and consider the so-called {\em tensor product} $G\times H$ \cite{Wan95tensor}, which is basically the same as the direct product of groups. Another possibility, however, is to liberate the C*-algebra multiplication and consider the {\em free product} $G*H$ \cite{Wan95free}. A natural question now is:
\begin{quote}
Are there any quantum groups between the tensor product and the free product?
\end{quote}

The difference between the quantum case and classical case is that we ``couple'' the quantum groups $G$ and $H$ not only using the group multiplication (here Woronowicz C*-algebra comultiplication) but also using the C*-algebra multiplication. In our work, we focus purely on non-commutative generalizations of the direct product. That is, we study products of quantum groups $G$ and $H$, where the factors $C(G)$ and $C(H)$ mutually cocommute, but not necessarily commute. So far the only results in this direction, which are known to us, are the above mentioned tensor and free products defined by Wang \cite{Wan95free,Wan95tensor}. Note that there are many results concerning generalizations of the group semidirect product \cite{Maj87,Maj89,VV03,BV05,MRW17} and wreath product \cite{Bic04,FS18}. Let us also mention here the glued product construction \cite{TW17,CW16}.

We focus in particular on extensions of orthogonal quantum groups $G\subset O_N^+$ by the group $\Z_2$. Our main tool for studying such extensions are categories of colored partitions. In the work of Banica and Speicher \cite{BS09}, certain combinatorial structures called categories of partitions are used to model the representation theory of a subclass of orthogonal quantum groups $S_N\subset G\subset O_N^+$, the so-called easy quantum groups. Using Woronowicz's Tannaka--Krein duality for compact quantum groups \cite{Wor88}, one is able to recover the compact matrix quantum group given the representation category of its fundamental representation. Thus, constructing examples of categories of partitions leads to examples of quantum groups. This approach was generalized by Freslon \cite{Fre17}, who introduced colored partitions to be able to describe more than one representation using partitions. We adapt this approach to the situation, where one of the representations is the one-dimensional representation of $\Z_2$. The structure we define in this way is called {\em category of partitions with extra singletons}. Examples of such structures induce examples of quantum groups $G$ such that $S_N\times E\subset G\subset O_N^+*\hat\Z_2$, where $E$ is the trivial group.

In Section \ref{sec.classification}, we reveal a correspondence between categories of partitions with extra singletons and two-colored categories of partitions that were used to describe unitary quantum groups \cite{TW17}.

\begin{customthm}{A}[Theorem \ref{T.F}]
\label{T.A}
The functor $F$ from Definition \ref{D.F} provides a one-to-one correspondence between categories of partitions with extra singletons of even length and categories of ``unitary'' two-colored partitions that are invariant with respect to the color inversions.
\end{customthm}

There are several classification results available for unitary two-colored partitions \cite{TW18,Gro18,MW18,MW19,MW19b}. Thanks to this correspondence, those classification results transfer to the case of partitions with extra singletons and hence provide many examples of $\Z_2$-extensions of orthogonal quantum groups. We summarize those classification results in Section \ref{sec.ex}. Closely related to our classification result is the work of Freslon \cite{Fre19}, see also Remark \ref{R.Fre}.


In Section \ref{sec.prods}, we return to the general question on the existence of quantum groups interpolating the free and tensor product. We bring a very general construction in Definition \ref{D.prods}. Given quantum groups $G=(C(G),u)$ and $H=(C(H),v)$, we define the following quantum subgroups of $G*H$. The product $G\ttimes H$ is defined by the relations
$$ab^*x=xab^*,\qquad a^*bx=xa^*b$$
the product $G\timess H$ by the relation
$$ax^*y=x^*ya,\qquad axy^*=xy^*a$$
the product $G\times_0 H$ by the combination of the both relations and finally, given $k\in\N$, the product $G\times_{2k}H$ by the relation
$$a_1x_1\cdots a_kx_k=x_1a_1\cdots x_ka_k,$$
where $a,b,a_1,\dots,a_k\in\{u_{ij}\}$ and $x,y,x_1,\dots,x_k\in\{v_{ij}\}$.

The following theorem then proves that we indeed construct something new lying strictly between the free and the tensor product.
\begin{customthm}{B}[Theorem \ref{T.prods}]
\label{T.B}
Consider quantum groups $G,H$. We have the following inclusions
$$G*H\kern1ex\begin{matrix}\lower.3ex\hbox{\rotatebox{15}{$\supset$}} & G\ttimes H & \rotatebox{345}{$\supset$}\\\raise.3ex\hbox{\rotatebox{345}{$\supset$}}  & G\timess H&\rotatebox{15}{$\supset$}\end{matrix}\kern1ex G\times_0 H\supset G\times_{2k}G\supset G\times_{2l}H\supset G\times_2H= G\times H,$$
where we assume $k,l\in\N$ such that $l$ divides $k$. The last three inclusions are strict if and only if the degree of reflection of both $G$ and $H$ is different from one.
\end{customthm}

The rest of Section \ref{sec.prods} is devoted to studying those products in the special case when $G$ is an orthogonal easy quantum group and $H=\hat\Z_2$ linking it to the previously developed theory of partitions with extra singletons.

Finally, let us mention the special case of quantum groups $G$ between $O_N\times E$ and $O_N^+*\Z_2$, which is studied in Subsection \ref{secc.pair}. Note that the group $O_N\times E$ is similar to the bistochastic group $B_{N+1}$ and the quantum group $O_N^+*\Z_2$ is similar to $B_{N+1}^{\#+}$. Thus, such instances $G$ with $O_N\times E\subset G\subset O_N^+*\Z_2$ are similar to quantum groups $\tilde G$ with $B_{N+1}\subset\tilde G\subset B_{N+1}^{\#+}$. The quantum groups $B_{N+1}$ and $B_{N+1}^{\#+}$ are easy quantum groups and correspond to the categories $\langle\singleton,\crosspart\rangle$, resp. $\langle \singleton\otimes\singleton\rangle$. However, not every quantum group lying between $B_{N+1}$ and $B_{N+1}^{\#+}$ is described by a category of partitions, i.e.\ not all of them are easy quantum groups. To describe also non-easy quantum groups, we have to deal with linear combinations of partitions as explained in \cite{GW18}. In particular, to interpret the examples of quantum groups $O_N\times E\subset G\subset O_N^+*\Z_2$ arising from partitions with extra singletons as quantum groups $B_{N+1}\subset \tilde G\subset B_{N+1}^{\#+}$ using classical partitions, we need to use the linear combination approach. The correspondence between partitions with extra singletons and linear combinations of ordinary partitions is described explicitly by the following theorem; see Definition \ref{D.U} for a definition of the functor $\U_{(N,\pm)}$.
\begin{customthm}{C}[Proposition \ref{P.U}, Theorem \ref{T.U}]
\label{T.C}
The map $\U_{(N,\pm)}\colon\Part\nlin\to\Part^{\tcol}\nnnlin$ is a monoidal unitary functor. It holds that
$$T_{\U_{(N,\pm)}p}=U^{\otimes l}_{(N,\pm)}T_pU^{*\,\otimes k}_{(N,\pm)}$$
for any $p\in\Part\nlin(k,l)$, $k,l\in\N_0$. Thus, considering a linear category of partitions $\Kat\subset\Part\nlin$ containing $\singleton\otimes\singleton$ and the corresponding quantum group $G$, it holds that the linear category with extra singletons $\U_{(N,\pm)}\Kat$ corresponds to the quantum group $\tilde G:=U_{(N,\pm)}GU_{(N,\pm)}^*$.
\end{customthm}

This explains the isomorphism between the quantum groups $O_N\times E\subset G\subset O_N^+*\hat\Z_2$ and $B_{N+1}\subset\tilde G\subset B_{N+1}^{\#+}$. As an application of Theorems \ref{T.A} and \ref{T.C}, we can take the classification result \cite{MW18,MW19,MW19b}, apply Theorem \ref{T.A} to obtain many examples of quantum groups $O_N\times E\subset G\subset O_N^+*\hat\Z_2$ and then apply Theorem \ref{T.C} to obtain many examples of non-easy quantum groups $B_{N+1}\subset \tilde G\subset B_{N+1}^{\#+}$.

\section{Preliminaries}
\label{sec.prelim}

In this section we recall the basic notions of compact matrix quantum groups and Tannaka--Krein duality. For a more detailed introduction, we refer to the monographs \cite{Tim08,NT13}.

\subsection{Compact matrix quantum groups}
\label{secc.qgdef}
Let $A$ be a C*-algebra, $u_{ij}\in A$, where $i,j=1,\dots,N$ for some $N\in\N$. Denote $u:=(u_{ij})_{i,j=1}^N\in M_N(A)$. The pair $(A,u)$ is called a \emph{compact matrix quantum group} if
\begin{enumerate}
\item the elements $u_{ij}$ $i,j=1,\dots, N$ generate $A$,
\item the matrices $u$ and $u^t=(u_{ji})$ are invertible,
\item the map $\Delta\colon A\to A\otimes_{\rm min} A$ defined as $\Delta(u_{ij}):=\sum_{k=1}^N u_{ik}\otimes u_{kj}$ extends to a $*$-homomorphism.
\end{enumerate}

Compact matrix quantum groups are generalizations of compact matrix groups in the following sense. For $G\subseteq M_N(\C)$ we can take the algebra of continuous functions $A:=C(G)$. This algebra is generated by the functions $u_{ij}\in C(G)$ assigning to each matrix $g\in G$ its $(i,j)$-th element $g_{ij}$. The so-called \emph{comultiplication} $\Delta\colon C(G)\to C(G)\otimes C(G)\simeq C(G\times G)$ is connected with the matrix multiplication on $G$ by $\Delta(f)(g,h)=f(gh)$ for $f\in C(G)$ and $g,h\in G$.

Therefore, for a general compact matrix quantum group $G=(A,u)$, the algebra $A$ should be seen as an algebra of non-commutative functions defined on some non-commutative compact underlying space. For this reason, we often denote $A=C(G)$ even if $A$ is not commutative. The matrix $u$ is called the \emph{fundamental representation} of $G$. Let us note that compact matrix quantum groups are special cases of compact quantum groups, see \cite{NT13,Tim08} for details.

A compact matrix quantum group $H=(C(H),v)$ is a \emph{quantum subgroup} of $G=(C(G),u)$, denoted as $H\subseteq G$, if $u$ and $v$ have the same size and there is a surjective $*$-homomorphism $\phi\colon C(G)\to C(H)$ sending $u_{ij}\mapsto v_{ij}$. We say that $G$ and $H$ are \emph{identical} if there exists such $*$-isomorphism (i.e. if $G\subset H$ and $H\subset G$).

One of the most important examples is the quantum generalization of the orthogonal group. The orthogonal group can be described as
$$O_N=\{U\in M_N(\C)\mid U_{ij}=\bar U_{ij},UU^t=U^tU=1_{\C^N}\}.$$
So, it can be treated also as a compact matrix quantum group $(C(O_N),u)$, where $C(O_N)$ can be described as a universal C*-algebra
$$C(O_N)=C^*(u_{ij},\;i,j=1,\dots,N\mid u_{ij}=u^*_{ij},uu^t=u^tu=1_{\C^N},u_{ij}u_{kl}=u_{kl}u_{ij}).$$

Such algebra can be quantized by dropping the commutativity relation. This was done by Wang in \cite{Wan95free} and the resulting algebra
$$C(O_N^+):=C^*(u_{ij},\;i,j=1,\dots,N\mid u_{ij}=u^*_{ij},uu^t=u^tu=1_{\C^N})$$
defines the \emph{free orthogonal quantum group}. In \cite{Wan95free} Wang also defined the \emph{free unitary quantum group} using the C*-algebra
$$C(U_N^+):=C^*(u_{ij},\;i,j=1,\dots,N\mid uu^*=u^*u=\bar uu^t=u^t\bar u=1_{\C^N}).$$

For a compact matrix quantum group $G=(C(G),u)$, we say that $v\in M_n(C(G))$ is a representation of $G$ if $\Delta(v_{ij})=\sum_{k}v_{ik}\otimes v_{kj}$, where $\Delta$ is the comultiplication. The representation $v$ is called \emph{unitary} if it is unitary as a matrix, i.e. $\sum_k v_{ik}v_{jk}^*=\sum_k v_{ki}^*v_{kj}=\delta_{ij}$. The fundamental representation is indeed a representation.

\subsection{The tensor and free product and their glued versions}
\label{secc.glued}

\begin{prop}[\cite{Wan95tensor}]
Let $G=(C(G),u)$ and $H=(C(H),v)$ be compact matrix quantum groups. Then $G\times H:=(C(G)\otimes_{\rm max}C(H),u\oplus v)$ is a compact matrix quantum group. For the comultiplication we have that
$$\Delta_\times(u_{ij}\otimes 1)=\Delta_G(u_{ij}),\quad\Delta_\times(1\otimes v_{kl})=\Delta_H(v_{kl}).$$
\end{prop}

The algebra $C(G)\otimes_{\rm max}C(H)$ can be described as a universal C*-algebra generated by elements $u_{ij}$ and $v_{kl}$ such that every $u_{ij}$ commutes with every $v_{kl}$, the elements $u_{ij}$ satisfy the same relations as $u_{ij}\in C(G)$ and the elements $v_{kl}$ satisfy the same relations as $v_{kl}\in C(H)$. Thus, the matrix 
$$u\oplus v=\begin{pmatrix}u&0\\ 0&v\end{pmatrix}\in M_n(C(G)\otimes_{\rm max}C(H))$$
indeed consists of generators of the algebra $C(G)\otimes_{\rm max}C(H)$. From now on, we will use such interpretation of the maximal tensor product and we will write just $u_{ij}v_{kl}$ instead of $u_{ij}\otimes v_{kl}$ without the explicit tensor sign.

A similar construction can be defined using the free product.

\begin{prop}[\cite{Wan95free}]
Let $G=(C(G),u)$ and $H=(C(H),v)$ be compact matrix quantum groups. Then $G*H:=(C(G)*_\C C(H),u\oplus v)$ is a compact matrix quantum group. For the comultiplication we have that
$$\Delta_*(u_{ij})=\Delta_G(u_{ij}),\quad\Delta_*(v_{kl})=\Delta_H(v_{kl}).$$
\end{prop}

We will call the quantum groups $G\times H$ and $G*H$ the \emph{tensor product} and the \emph{free product} of $G$ and $H$ respectively.

%

\begin{prop}[\cite{TW17}] Let $G=(C(G),u)$ and $H=(C(H),v)$ be compact matrix quantum groups. Let $A$ be the C*-subalgebra of $C(G)\otimes_{\rm max} C(H)$ generated by the products $u_{ij}v_{kl}$, i.e. generated by the elements of the matrix $u\otimes v$. Then the \emph{glued tensor product} $G\tiltimes H:=(A,u\otimes v)$ is a compact matrix quantum group. For the comultiplication we have that
$$\Delta_{\tiltimes}(u_{ij}v_{kl})=\Delta_G(u_{ij})\Delta_H(v_{kl})$$
\end{prop}

Similarly, one can define the \emph{glued free product} $G\tilstar H$, which we will not use in this article.


\subsection{Monoidal involutive categories and Tannaka--Krein duality}
Let $R$ be a set of \emph{objects}. For every $r,s\in R$, let $\Mor(r,s)$ be a set of \emph{morphisms} between $r$ and $s$. Let us have associative binary operations $\otimes\colon R\times R\to R$ and $\otimes\colon\Mor(r,s)\times\Mor(r',s')\to\Mor(r\otimes r',s\otimes s')$. Let $\cdot:\Mor(r,s)\otimes\Mor(p,r)\to\Mor(p,s)$ be another associative binary operation. Finally, let $*$ be an involution mapping $\Mor(r,s)\to\Mor(s,r)$. Then the tuple $(R,\{\Mor(r,s)\}_{r,s\in R},\otimes,\cdot,*)$ forms a (small strict) \emph{monoidal involutive category} if the following additional conditions hold
\begin{itemize}
\item For every $r\in R$, there is an identity $1_r\in\Mor(r,r)$ satisfying $1\cdot T_1=T_1$ and $T_2\cdot 1=T_2$ for every $T_1\in\Mor(p,r)$ and every $T_2\in\Mor(r,s)$.
\item There is $1\in R$ such that, for every $r\in R$, $1\otimes r=r\otimes 1=r$.
\end{itemize}
If the sets of morphisms have the structure of a vector space and the corresponding maps of morphisms are linear (or antilinear in the case of involution), we call the structure a \emph{monoidal $*$-category}.

A monoidal $*$-category is called \emph{concrete} if the morphisms are realized by matrices. That is, there is a map $n\colon R\to\N_0$ such that $\Mor(r,s)\subset\Lin(\C^{n(r)},\C^{n(s)})$ and the operations are defined the standard way (the composition is realized by matrix multiplication, the tensor product by the Kronecker product, the involution by conjugate transpose, the identity morphism is the identity matrix).

%

For given two objects $r,s$ of a monoidal category, we will say that they are \emph{dual} to each other (denoted $s=\bar r$ or $r=\bar s$) if there are morphisms $T_1\in\Mor(1,r\otimes s)$ and $T_2\in\Mor(1,s\otimes r)$ such that $(T_1^*\otimes 1_r)(1_r\otimes T_2)=1_r$ and $(T_2^*\otimes 1_s)(1_s\otimes T_1)=1_s$. A monoidal category, where all objects have their dual is called a monoidal category \emph{with duals}.

An important example of a monoidal $*$-category with duals is the set of all unitary representations $\Rep G$ of a given compact matrix quantum group $G$, where the set of morphisms between two representations $u$ and $v$ is the space of intertwiners $\Mor(u,v)$. A dual of a representation $u=(u_{ij})$ is simply its complex conjugate $\bar u=(u_{ij}^*)$. Such a category is, in addition, \emph{complete} in the sense that it is closed under taking equivalent objects, subobjects and direct sums of objects.

For every compact matrix quantum group $G=(C(G),u)$, it was shown \cite{Wor87} that all representations are direct sums of irreducible ones and that any irreducible representation $v$ is contained as a subrepresentation in a tensor product of sufficiently many copies of the fundamental representation $u$ and its complex conjugate $\bar u$. Thus, to describe the representation theory of a given quantum group $G$, it is enough to consider the category $\Reptil G$ of representations that are made as a tensor product of copies of $u$ and $\bar u$. The complete category $\Rep G$ can be computed as the natural completion of $\Reptil G$.

One of the most important results for compact quantum groups is the Tannaka--Krein duality that was proven by Woronowicz in \cite{Wor88}. It says that conversely given a concrete monoidal $*$-category $R$ generated by some object $r$ and its dual $\bar r$, there exists a compact matrix quantum group $G$ such that $\Rep G$ is the completion of $R$.

\subsection{Partitions}

Let $k,l\in\N_0$, by a \emph{partition} of $k$ upper and $l$ lower points we mean a partition of the set $\{1,\dots,k\}\sqcup\{1,\dots,l\}\approx\{1,\dots,k+l\}$. That is, a decomposition of the set of $k+l$ points into non-empty disjoint subsets called \emph{blocks}. The first $k$ points are called \emph{upper} and the last $l$ points are called \emph{lower}. The set of all partitions on $k$ upper and $l$ lower points is denoted $\Part(k,l)$. We denote the union $\Part:=\bigcup_{k,l\in\N_0}\Part(k,l)$. The number $\left| p\right|:=k+l$ for $p\in\Part(k,l)$ is called the \emph{length} of $p$.

We illustrate partitions graphically by putting $k$ points in one row and $l$ points in another row below and connecting by lines those points that are grouped in one block. All lines are drawn between those two rows.

Below, we give an example of two partitions $p\in \Part(3,4)$ and $q\in\Part(4,4)$ including their graphical representation. The first set of points is decomposed into three blocks, whereas the second one into four blocks. In addition, the first one is an example of a \emph{non-crossing} partition, i.e. a partition that can be drawn in a way that lines connecting different blocks do not intersect (following the rule that all lines are between the two rows of points). On the other hand, the second partition has one crossing.

\begin{equation}
\label{eq.pq}
p=
\BigPartition{
\Pblock 0 to 0.25:2,3
\Pblock 1 to 0.75:1,2,3
\Psingletons 0 to 0.25:1,4
\Pline (2.5,0.25) (2.5,0.75)
}
\qquad
q=
\BigPartition{
\Psingletons 0 to 0.25:1,4
\Psingletons 1 to 0.75:1,4
\Pline (2,0) (3,1)
\Pline (3,0) (2,1)
\Pline (2.75,0.25) (4,0.25)
}
\end{equation}

A block containing a single point is called a \emph{singleton}. In particular, the partitions containing only one point are called singletons and for the sake of clarity denoted by an arrow $\singleton\in\Part(0,1)$ and $\upsingleton\in\Part(1,0)$.

\subsection{Categories of partitions}
\label{secc.cat}

We define the following operations on $\Part$.
\begin{itemize}
\item  The \emph{tensor product} of two partitions $p\in\Part(k,l)$ and $q\in\Part(k',l')$ is the partition $p\otimes q\in \Part(k+k',l+l')$ obtained by writing the graphical representations of $p$ and $q$ ``side by side''
$$
\BigPartition{
\Pblock 0 to 0.25:2,3
\Pblock 1 to 0.75:1,2,3
\Psingletons 0 to 0.25:1,4
\Pline (2.5,0.25) (2.5,0.75)
}
\otimes
\BigPartition{
\Psingletons 0 to 0.25:1,4
\Psingletons 1 to 0.75:1,4
\Pline (2,0) (3,1)
\Pline (3,0) (2,1)
\Pline (2.75,0.25) (4,0.25)
}
=
\BigPartition{
\Pblock 0 to 0.25:2,3
\Pblock 1 to 0.75:1,2,3
\Psingletons 0 to 0.25:1,4,5,8
\Psingletons 1 to 0.75:5,8
\Pline (2.5,0.25) (2.5,0.75)
\Pline (6,0) (7,1)
\Pline (7,0) (6,1)
\Pline (6.75,0.25) (8,0.25)
}
$$

\item For $p\in\Part(k,l)$, $q\in\Part(l,m)$ we define their \emph{composition} $qp\in\Part(k,m)$ by putting the graphical representation of $q$ bellow $p$ identifying the lower row of $p$ with the upper row of $q$. The upper row of $p$ now represents the upper row of the composition and the lower row of $q$ represents the lower row of the composition. By the vertical concatenation, there may appear certain strings that are connected neither to one of the upper or the lower points of the result. Those are removed and the number of such \emph{loops} is denoted by $\rl(p,q)$.
$$
\BigPartition{
\Psingletons 0 to 0.25:1,4
\Psingletons 1 to 0.75:1,4
\Pline (2,0) (3,1)
\Pline (3,0) (2,1)
\Pline (2.75,0.25) (4,0.25)
}
\cdot
\BigPartition{
\Pblock 0 to 0.25:2,3
\Pblock 1 to 0.75:1,2,3
\Psingletons 0 to 0.25:1,4
\Pline (2.5,0.25) (2.5,0.75)
}
=
\BigPartition{
\Pblock 0.5 to 0.75:2,3
\Pblock 1.5 to 1.25:1,2,3
\Psingletons  0.5 to  0.75:1,4
\Pline (2.5,0.75) (2.5,1.25)
\Psingletons -0.5 to -0.25:1,4
\Psingletons  0.5 to  0.25:1,4
\Pline (2,-0.5) (3,0.5)
\Pline (3,-0.5) (2,0.5)
\Pline (2.75,-0.25) (4,-0.25)
}
=
\BigPartition{
\Pblock 0 to 0.25:2,3,4
\Pblock 1 to 0.75:1,2,3
\Psingletons 0 to 0.25:1
\Pline (2.5,0.25) (2.5,0.75)
}
$$

\item For $p\in\Part(k,l)$ we define its \emph{involution} $p^*\in\Part(l,k)$ by reversing its graphical representation with respect to the horizontal axis.
$$
\left(
\BigPartition{
\Pblock 0 to 0.25:2,3
\Pblock 1 to 0.75:1,2,3
\Psingletons 0 to 0.25:1,4
\Pline (2.5,0.25) (2.5,0.75)
}
\right)^*
=
\BigPartition{
\Pblock 1 to 0.75:2,3
\Pblock 0 to 0.25:1,2,3
\Psingletons 1 to 0.75:1,4
\Pline (2.5,0.25) (2.5,0.75)
}
$$
\end{itemize}

These operations are called the \emph{category operations} on partitions.

The set of all natural numbers with zero $\N_0$ as a set of objects together with the sets of partitions $\Part(k,l)$ as sets of morphisms between $k\in\N_0$ and $l\in\N_0$ with respect to those operations form a monoidal involutive category. All objects in the category are self-dual.

Any collection of subspaces $\Cat=\bigcup_{k,l\in\N_0}\Cat(k,l)$, where $\Cat(k,l)\subset\Part(k,l)$, containing the \emph{identity partition} $\idpart\in\Cat(1,1)$ and the \emph{pair partition} $\pairpart\in\Cat(0,2)$ and closed under the category operations is a monoidal involutive category with duals. We call it a \emph{category of partitions}. (In subsequent sections, we will sometimes refer to this object as an \emph{ordinary} or \emph{one-colored} category of partitions to distinguish it from some generalizations.)

For given $p_1,\dots,p_n\in\Part\nlin$ we denote by $\langle p_1,\dots,p_n\rangle$ the smallest category of partitions containing $p_1,\dots,p_n$. We say that $p_1,\dots,p_n$ generate $\langle p_1,\dots,p_n\rangle$. Note that the pair partitions are contained in the category by definition and hence will not be explicitly listed as generators.

\subsection{Linear maps associated to partitions}
\label{secc.Tp}
Now, consider a fixed natural number $N\in\N$. Given a partition $p\in\Part(k,l)$, we can define a linear map $T_p\colon(\C^N)^{\otimes k}\to(\C^N)^{\otimes l}$ via
\begin{equation}
\label{eq.Tp}
T_p(e_{i_1}\otimes\cdots\otimes e_{i_k})=\sum_{j_1,\dots,j_l=1}^N\delta_p(\mathbf{i},\mathbf{j})(e_{j_1}\otimes\cdots\otimes e_{j_l}),
\end{equation}
where $\mathbf{i}=(i_1,\dots,i_k)$, $\mathbf{j}=(j_1,\dots,j_l)$ and the symbol $\delta_p(\mathbf{i},\mathbf{j})$ is defined as follows. Let us assign to the $k$ points in the upper row of $p$ the numbers $i_1,\dots,i_k$ (from left to right) and to the $l$ points in the lower row $j_1,\dots,j_l$ (again from left to right). Then $\delta(\mathbf{i},\mathbf{j})=1$ if the points belonging to the same block are assigned the same numbers. Otherwise $\delta(\mathbf{i},\mathbf{j})=0$. 

As an example, we can express $\delta_p$ and $\delta_q$, where $p$ and $q$ come from Equation \eqref{eq.pq}, using the multivariate $\delta$ function as follows
$$\delta_p(\mathbf{i},\mathbf{j})=\delta_{i_1i_2i_3j_2j_3},\quad
\delta_q(\mathbf{i},\mathbf{j})=\delta_{i_2j_3j_4}\delta_{i_3j_2}.$$

Given a partition $p\in\Part(k,l)$, we can interpret the map $T_p$ as an intertwiner $T_pu^{\otimes k}=u^{\otimes l}T_p$ for some compact matrix quantum group $G$. Substituting the definition of $T_p$, this implies the following relations
\begin{equation}
\label{eq.Trel}
\sum_{t_1,\dots,t_k=1}^N\delta_p(\mathbf{t},\mathbf{s})u_{t_1i_1}\cdots u_{t_ki_k}=\sum_{j_1,\dots,j_l=1}^N\delta_p(\mathbf{i},\mathbf{j})u_{s_1j_1}\cdots u_{s_lj_l}
\end{equation}
for every $i_1,\dots,i_k,s_1,\dots,s_l\in\{1,\dots,N\}$.

For example, considering $p=\pairpart\in\Part(0,2)$, we have the relation $$\delta_{s_1s_2}=\sum_{j=1}^Nu_{s_1j}u_{s_2j}.$$ Thus, for any quantum group $G\subset O_N^+$, we have that $T_{\pairpart}\in\Mor(1,u\otimes u)$. Similarly, we also have $T_{\uppairpart}\in\Mor(u\otimes u,1)$ for any $G\subset O_N^+$.

The map $p\mapsto T_p$ is almost a monoidal unitary functor in the sense that we have
\begin{itemize}
\item $T_{p^*}=T_p^*$,
\item $T_{p\otimes q}=T_p\otimes T_q$ and
\item $T_{qp}=N^{-\rl(p,q)}T_qT_p$.
\end{itemize}
Note that we can make the map $T_\bullet$ an honest monoidal unitary functor if we define a linear structure on the sets of partitions and modify the multiplication by incorporating the factor $N^{\rl(p,q)}$. See Section \ref{sec.linear}.

As a consequence, given a category of partitions $\Cat$, the collection of sets $\spanlin\{T_p\mid p\in\Cat(k,l)\}$ for $k,l\in\N_0$ forms a concrete monoidal $*$-category. According to Tannaka--Krein duality, this means that it can be considered as a category of representations of some compact matrix quantum group $G$. A quantum group that arise by such a construction is called an (orthogonal) \emph{easy} quantum group \cite{BS09}.

We can express the corresponding compact matrix quantum group $G$ very concretely using universal C*-algebras as
$$C(G)=C^*(u_{ij},\;i,j=1,\dots,N\mid u=\bar u,\; T_pu^{\otimes k}=u^{\otimes l}T_p\;\forall p\in\Cat(k,l),\;k,l\in\N_0).$$
Suppose that a set $S$ generates $\Cat$. Then, thanks to the functorial property of $T$, we have
$$C(G)=C^*\left(u_{ij},\;i,j=1,\dots,N\mathbin\Big|\begin{matrix} u=\bar u,\; uu^t=u^tu=1_{\C^N},\\ T_pu^{\otimes k}=u^{\otimes l}T_p\;\forall p\in S(k,l),\;k,l\in\N_0\end{matrix}\right).$$
Recall that $T_pu^{\otimes k}=u^{\otimes l}T_p$ amounts to very concrete relations on the generators $u_{ij}$, see Equation \eqref{eq.Trel}.

\section{Colored categories of partitions}
\label{sec.col}
In \cite{Fre17}, Freslon introduced a formalism of coloured partitions to describe more general quantum groups. In the following subsection, we briefly summarize the general formalism. Then we describe concrete applications. Note that colored partitions were already used to describe unitary quantum groups \cite{TW17} or wreath product of quantum groups \cite{FS18}. However, we will use it in a bit more primitive way to describe quantum groups with fundamental representation in the form of a direct sum.

\subsection{Colored partitions}
\label{secc.col}
Let $\Alphabet$ be a finite alphabet equipped with an involution $x\mapsto\bar x$. An $\Alphabet$-colored partition is a triple $(p,w_1,w_2)$, where $p\in\Part(k,l)$, $w_1$ is a word of length $k$ and $w_2$ is a word of length $l$ over the alphabet $\Alphabet$. The word $w_1$ is called the \emph{upper color pattern} and $w_2$ is called the \emph{lower color pattern} of the partition. We denote $\Part^\Alphabet(w_1,w_2)$ the set of all $\Alphabet$-colored partitions with color patterns $w_1$ and $w_2$. Again, we denote $\Part^\Alphabet(k,l)=\bigcup_{\substack{|w_1|=k,\\|w_2|=l}}\Part^\Alphabet(w_1,w_2)$ and $\Part^\Alphabet=\bigcup_{k,l\in\N_0}\Part^{\Alphabet}(k,l)$.

We define the operations of tensor product, composition and involution the same way as for ordinary partitions (Subsection \ref{secc.cat}) except for the following. The composition $qp$ of $p\in\Part^\Alphabet(w_1,w_2)$ and $q\in\Part^\Alphabet(w_1',w_2')$ is defined only if $w_1'=w_2$.

A subset $\Cat\subset\Part^\Alphabet$ is called a \emph{category of $\Alphabet$-colored partitions} if it is closed under those operations and contains all the identity partitions $(\idpart,x,x)$ and all the pair partitions $(\pairpart,\emptyset,x\bar x)$ for all $x\in\Alphabet$. Any category of $\Alphabet$-colored partitions indeed forms a monoidal involutive category with duals.


%

\subsection{Two-colored ``unitary'' partitions}
We get an important instance of colored categories of partitions if $\Alphabet=\{\wcol,\bcol\}$ and $\bar\wcol=\bcol$. The set of all such two-colored partitions will be denoted $\Part\twocol$. Categories of such two-colored partitions correspond to unitary quantum groups in the following way.

Let us take a category $\Cat\subset\Part\twocol$ and a natural number $N\in\N$. We construct a quantum group $G=(C(G),u)$, where $S_N\subset G\subset U_N^+$ as follows. The color $\wcol$ corresponds to the fundamental representation $u$. The color $\bcol$ cooresponds to its complex conjugate $\bar u$. Thus any object in $\Cat$, i.e. a word $w=a_1\cdots a_k$ over the alphabet $\{\wcol,\bcol\}$ corresponds to a $k$-fold tensor product $u^{\otimes w}:=u^{a_1}\otimes\cdots\otimes u^{a_k}$, where $u^\wcol:=u$ and $u^\bcol:=\bar u$. Any partition $p\in\Cat(w_1,w_2)$ is assigned a linear map $T_p\colon (\C^N)^{\otimes |w_1|}\to(\C^N)^{\otimes |w_2|}$ given by Equation \eqref{eq.Tp}. Such a map should be interpreted as an intertwiner between $u^{\otimes w_1}$ and $u^{\otimes w_2}$.

To be more precise, any category $\Cat\subset\Part\twocol$ induces a concrete monoidal $*$-category, where objects are words over the alphabet $\{\wcol,\bcol\}$ and the morphism spaces are given by $\Mor(w_1,w_2)=\spanlin\{T_p\mid p\in\Cat(w_1,w_2)\}$. According to Tannaka--Krein duality there exists a quantum group $G=(C(G),u)$, where the intertwiner spaces are given by this category $\Mor(u^{\otimes w_1},u^{\otimes w_2})=\Mor(w_1,w_2)$. Such a quantum group $G$ is called a (unitary) \emph{easy} quantum group \cite{TW17} and can be constructed as a universal C*-algebra 
$$C(G)=C^*(u_{ij},\;i,j=1,\dots,N\mid T_pu^{\otimes w_1}=u^{\otimes w_2}T_p\;\forall p\in\Cat(w_1,w_2)).$$

For more information about this correspondence, see \cite{Fre17,TW17}. Note also some classification results of the unitary partition categories \cite{TW18,Gro18,MW18,MW19,MW19b}.

\subsection{Two independent colors}
\label{secc.indep}
Consider an alphabet $\Alphabet$ and a partition $p\in\Part^\Alphabet$. We say that two colors of $\Alphabet$ are \emph{independent} in $p$ if no block of $p$ contains both of these colors.

Consider for simplicity two self-dual colors $\Alphabet=\{\wcol,\lozenge\}$ and fix two natural numbers $N_\wcol,N_\lozenge\in\N$. Then any category $\Cat\subset\Part^\Alphabet$ containing only partitions where the two colors are independent can be assigned a quantum group $G=(C(G),u)$, where $S_{N_\wcol}\times S_{N_\lozenge}\subset G\subset O_{N_\wcol}^+* O_{N_\lozenge}^+$.

The color $\wcol$ corresponds to some representation $u^\wcol$ and the color $\lozenge$ corresponds to some representation $u^\lozenge$. Thus, the words $w$ over $\Alphabet$ correspond to tensor products $u^{\otimes w}$. A partition $p\in\Cat(w_1,w_2)$, where $w_1=a_1\cdots a_k$ and $w_2=b_1\cdots b_l$ then corresponds to a map $T_p\colon \C^{N_{a_1}}\otimes\cdots\otimes\C^{N_{a_k}}\to\C^{N_{b_1}}\otimes\cdots\otimes\C^{N_{a_l}}$ given again by Equation \ref{eq.Tp} (where the summation for each $j_n$ goes from 1 to $N_{a_n}$). Then, we can construct a quantum group $G=(C(G),u)$, where $u=u^\wcol\oplus u^\lozenge$ and
$$C(G)=C^*(u_{ij}^\wcol,u_{ij}^\lozenge\mid u^\wcol=\bar u^\wcol,\;u^\lozenge=\bar u^\lozenge,\; T_pu^{\otimes w_1}=u^{\otimes w_2}T_p\;\forall p\in\Cat(w_1,w_2)).$$

To assure that such quantum group indeed exists, we again use the Tannaka--Krein theorem for quantum groups. Let us denote $N:=N_\wcol+N_\lozenge$. Note that using the canonical projections $\C^N\to\C^{N_\wcol}$, $\C^N\to\C^{N_\lozenge}$ and the corresponding embeddings, we can interpret the maps $T_p$ as a mapping $(\C^N)^{\otimes k}\to (\C^N)^{\otimes l}$ and hence as intertwiners for $u=\bar u$.

Note that we have $G=O_{N_\wcol}^+*O_{N_\lozenge}^+$ for $\Cat$ consisting of all non-crossing pair partitions with $\wcol$ and $\lozenge$ being independent. Likewise, $G=S_{N_\wcol}\times S_{N_\lozenge}$ for $\Cat$ consisting of all partitions with $\wcol$ and $\lozenge$ independent.

\begin{rem}
In the special case $N:=N_\wcol=N_\lozenge$ we can consider really all partitions over $\Alphabet$ without assuming that the colors are independent. This allows us to somehow \emph{amalgamate} the two factors. The resulting quantum group $G$ will then sit between $S_N\subset G\subset O_N^+*O_N^+$, where $S_N$ is taken as a subgroup of $S_N\times S_N$ by identifying the two factors. This is the approach originally formulated by Freslon in \cite{Fre17}. Non-crossing categories on two self-dual colors were classified in \cite{Fre19}.
\end{rem}

\section{Partition quantum groups with one-dimensional factor}

\subsection{Partitions with extra singletons}
In this article, we are interested in describing quantum groups $G$, whose fundamental representation $u$ decomposes as a direct sum $v\oplus r$ of an $N$-dimensional representation $v\in M_N(A)$ and a one-dimensional representation $r\in A$. So, let us consider a two-letter alphabet $\Alphabet=\{\tcol,\lincol\}$. The triangle $\tcol$ corresponds to the representation $r$ and the line $\lincol$ corresponds to the representation $v$.

Since the representation $r$ is one-dimensional, one can see that the block structure of the color $\tcol$ is irrelevant, because it does not affect the corresponding map $T_p$ at all. This motivates the following definition.

\begin{defn}
Consider the alphabet $\Alphabet=\{\tcol,\lincol\}$ with the trivial involution $\tcol\mapsto\tcol$, $\lincol\mapsto\lincol$. A \emph{partition with extra singletons} is an $\Alphabet$-colored partition $p$, where all points of the color $\tcol$ are singletons. Any set $\Cat$ of partitions with extra singletons that is closed under the category operations and contains the partitions $\idpart$, $\idext$, $\pairpart$ and $\singext\otimes\singext$ is called a \emph{category of partitions with extra singletons}. The category of all partitions with extra singletons is denoted $\Part^{\tcol}$. The points with color $\tcol$ are called \emph{extra singletons}. The smallest category of partitions with extra singletons containing given $p_1,\dots,p_n\in\Part^{\tcol}$ is denoted by $\langle p_1,\dots,p_n\rangle^{\tcol}$.
\end{defn}

Recall from Subsection \ref{secc.col} that the composition $qp$ of two partitions $p$, $q$ with extra singletons is defined only if the upper color pattern of $q$ matches the lower color pattern of $p$, that is, extra singletons can be composed only with extra singletons.

A typical example of partition with extra singletons looks as follows
\begin{equation}
\label{eq.pext}
p=\BigPartition{
\Psingletons 0to0.25:5
\Psingletons 1to0.75:1
\Ppoint0 \Lt:1,4
\Ppoint1 \Ut:3,4
\Pblock 0to0.3:2,3
\Pblock 1to0.7:2,5
\Pline(2.5,0.3)(2.5,0.7)
}.
\end{equation}

%
%

\subsection{The corresponding quantum groups}
Let us summarize here the meaning of such categories by applying the general considerations mentioned in Sections \ref{sec.prelim} and \ref{sec.col}. A category of partitions with extra singletons, being a colored category with independent colors, corresponds to some quantum group $G$ with $S_N\times S_{N'}\subset G\subset O_N^+*O_{N'}^+$. As we mentioned above, we will always assume $N'=1$, so actually we have
$$S_N\times E\subset G\subset O_N^+*\hat\Z_2,$$
where $E=(\C,1)$ is the trivial (compact matrix) quantum group. So, $G$ is a matrix quantum group with matrix of size $(N+1)\times(N+1)$ having a block structure with one block of size $N$ and second block of size one. We will usually denote the fundamental representation of $G$ by $u=v\oplus r$.

To define this quantum group, we have to first describe the $T_p$ maps. In our case, they can be defined as follows. Consider a partition with extra singletons $p\in\Part^{\tcol}$. Denote by $k'$ resp.\ $l'$ the number of upper resp.\ lower points of the color $\lincol$ (i.e.\ not being extra singletons). Then we define $T_p\colon(\C^N)^{\otimes k'}\to(\C^N)^{\otimes l'}$ by Equation \eqref{eq.Tp} ignoring all the extra singletons in $p$. The extra singletons become important when we interpret the partition $p$ as an intertwiner $T_pu^{\otimes w_1}=u^{\otimes w_2}T_p$, where $w_1$ and $w_2$ are the upper and the lower color pattern of $p$, respectively.

\begin{ex}
Given a partition $\positionerext$, it is associated a map $T_{\positionerext}\colon\C^N\to\C^N$, which coincides with the map associated to the identity partition $\idpart$. It is the identity $\C^N\to\C^N$. However, the interpretation of those partitions are different. While the identity partition gives us just the trivial relation
$$v=vT_{\idpart}=T_{\idpart}v=v,$$
the relation associated to the partition $\positionerext$ reads
$$vr=v\otimes r=T_{\positionerext}(v\otimes r)=(r\otimes v)T_{\positionerext}=r\otimes v=rv.$$

See Example \ref{ex.FTp} for another example of a $T_p$ map and a relation associated to a partition with extra singletons.
\end{ex}

Now a quantum group $G=(C(G),v\oplus r)$ corresponding to a given category of partitions with extra singletons $\Cat\subset\Part^{\tcol}$ can be defined by
$$C(G)=C^*(v_{ij},r\mid v_{ij}=v_{ij}^*,\;r=r^*,\; T_pu^{\otimes w_1}=u^{\otimes w_2}T_p\;\forall p\in\Cat(w_1,w_2)),$$
where $u^{\lincol}=v$ and $u^{\tcol}=r$.

As we mentioned at the end of Subsection \ref{secc.Tp}, we do not have to consider the relations corresponding to all the partitions in $\Cat$, but only to some generating set of $\Cat$. So, suppose $\Cat=\langle S\rangle^{\tcol}$, then we have
$$C(G)=C^*\left(v_{ij},r\mathbin\Big|\begin{matrix} v=\bar v,\; vv^t=v^tv=1_{\C^N},r=r^*,\;r^2=1\\ T_pu^{\otimes w_1}=u^{\otimes w_2}T_p\;\forall p\in S(w_1,w_2)\end{matrix}\right).$$
Note that the orthogonality relations $vv^t=v^tv=1$ and $r^2=1$ correspond to the partitions $\pairpart$ and $\singext\otimes\singext$, which are contained in any category by definition, but they are usually not explicitly listed as generators.

\subsection{The quantum group relations}
Let us now give a few examples of partitions with extra singletons and the corresponding quantum group relations. Recall that we must assume that all the generators $v_{ij}$ and $r$ are self-adjoint -- this does not follow from any partition relation. We have the following correspondences:
\begin{align*}
\pairpart                       & \qquad vv^t=1\\
\singext\otimes\singext         & \qquad r^2=1\\
\singext                        & \qquad r=1\\
\positionerext                  & \qquad v_{ij}r=rv_{ij}\\
\globcolext                     & \qquad v_{ij}v_{kl}r=rv_{ij}v_{kl}\\
\end{align*}

The first two partitions correspond to orthogonality of $u$ and $r$ and are by definition present in all categories with extra singletons. The following partitions allow us to construct the most basic instances of extra-singleton categories and the corresponding quantum groups.

\begin{prop}
\label{P.prods}
Let $\Cat\subset\Part$ be an ordinary category of partitions corresponding to a quantum group $H\subset O_N^+$. Then
\begin{enumerate}
\item $\langle\Cat\rangle^{\tcol}$ corresponds to $H*\hat\Z_2$,
\item $\langle\Cat,\positionerext\rangle^{\tcol}$ corresponds to $H\times\hat\Z_2$,
\item $\langle\Cat,\singext\rangle^{\tcol}$ corresponds to $H\times E$,
\end{enumerate}
where $E=(\C,1)$ is the trivial (quantum) group.
\end{prop}
\begin{proof}
We just need to look at the relations implied by the generators of the categories and find out which quantum subgroup $G\subset O_N^+*\hat\Z_2$ they determine. In the first case, we only have partitions without extra singletons, so in the corresponding relations only the elements $v_{ij}$ appear (not the subrepresentation $r$). In particular, those relations correspond to the subgroup $H\subset O_N^+$, so, taken as generators of a category with extra singletons, they define the subgroup $H*\hat\Z_2\subset O_N^+*\hat\Z_2$. Indeed, since we do not add any new relations on $r$, the $r$ remains free from $v_{ij}$ and keeps representing the factor $\hat\Z_2$.

In the second case, we have, in addition, the generator $\positionerext$, which corresponds to the relation $v_{ij}r=rv_{ij}$. Thus, the category corresponds to the quantum subgroup of $H*\hat\Z_2$ given by this relation. The relation is simply commutativity of the factors $C(H)$ and $C^*(\Z_2)$, so the corresponding quantum group is the tensor product $H\times\hat\Z_2$.

Finally, the last instance corresponds to the subgroup of $H*\hat\Z_2$ with respect to the relation $r=1$. This relation corresponds to taking just the trivial subgroup of~$\hat\Z_2$.
\end{proof}

The example $\globcolext$ corresponding to some weaker kind of commutativity can be seen as a motivating example for our article. We are going to show (see Proposition \ref{P.prodtab}) that for any category of partitions $\Cat\subset\Part$ we have that
$$\langle\Cat\rangle^{\tcol}\subsetneq\langle\Cat,\globcolext\rangle^{\tcol}\subsetneq\langle\Cat,\positionerext\rangle^{\tcol}.$$
Such a category hence corresponds to a quantum group $G$ with
$$H*\hat\Z_2\supsetneq G\supsetneq H\times\hat\Z_2,$$
which can be seen as a new kind of quantum group product of $H$ with $\hat\Z_2$. A similar result can be formulated with many other partitions with extra singletons (see Section \ref{secc.easy}).

\subsection{Induced one-colored categories}
In the preceding subsection we studied the most simple constructions of how to get an extra-singleton category from an ordinary category. Here, we are going to study the converse. Given a category of partitions with extra singletons $\Cat\subset\Part^{\tcol}$ corresponding to a quantum group $G$, we are going to define two natural ordinary categories associated to $\Cat$. The first one corresponds to the smallest quantum group $H\subset O_N^+$ such that $G\subset H*\hat\Z_2$. The second one corresponds to the largest quantum subgroup $\tilde H\subset O_N^+$ contained in $G$ (in the sense $\tilde H\times E\subset G$, where $E$ is the trivial group). Note that we may have $H\neq\tilde H$ in contrast with the simple examples of the previous subsection.

Firstly, the set of all one-colored partitions $\Part$ can be viewed as a subset of $\Part^{\tcol}$. In this sense, every category of partitions with extra singletons $\Cat$ induces a category of one-colored partitions by restriction:
$$\Cat^{\lincol}:=\{p\in\Cat\mid\text{$p$ does not contain any extra singleton}\}.$$
It is easy to check that $\Cat^{\lincol}$ is a category of partitions for any category of partitions with extra singletons $\Cat$.

\begin{lem}
\label{L.H}
Let $\Cat$ be a category of partitions with extra singletons corresponding to a quantum group $G=(C(G),v\oplus r)$. Then the one-colored category $\Cat^{\lincol}$ corresponds to the quantum group $H=(A,v)$, where $A\subset C(G)$ is the subalgebra generated by $\{v_{ij}\}_{i,j=1}^N$.
\end{lem}
\begin{proof}
We are looking for a compact matrix quantum group $H$ whose intertwiner spaces are given by the partitions in $\Cat^{\lincol}$. Defining $H$ as in the statement of the lemma, we first need to prove that it is a compact matrix quantum group. This is easy to check: (1) $A$ is generated by $v_{ij}$ by definition, (2) if a block diagonal matrix has an inverse, then the blocks must also be invertible, (3) the comultiplication is given simply by restriction to $A$. Finally, by definition of $G$, the elements of $\Cat^{\lincol}$ precisely describe the intertwiners of $v$ as a subrepresentation of $u$. We defined $H$ in such a~way that $v$ is its fundamental representation, so its intertwiners are indeed described by $\Cat^{\lincol}$.
\end{proof}

Secondly, given a category of partitions with extra singletons $\Cat$, we can somehow ignore the extra singletons. Let us define the following
$$\Cat^{\tcol=1}:=(\bar\Cat)^{\lincol},$$
where $\bar\Cat=\langle\Cat,\singext\rangle^{\tcol}$. Obviously, we have $\Cat^{\lincol}\subset\Cat^{\tcol=1}$.


\begin{lem}
\label{L.tilH}
Let $\Cat$ be a category of partitions with extra singletons corresponding to a quantum group $G=(C(G),v\oplus r)$. Then the one-colored category $\Cat^{\tcol=1}$ corresponds to the quantum group $\tilde H=(A,\tilde v)$, where $A$ is the quotient of $C(G)$ by the relation $r=1$ and $\tilde v_{ij}$ are the images of $v_{ij}$ under the natural homomorphism.
\end{lem}
\begin{proof}
The partition $\singext$ correspons to the relation $r=1$, so $\bar\Cat$ is the category corresponding to the quantum subgroup $\bar G\subset G$ defined by imposing the relation $r=1$. That is, $\bar G=(A,\tilde v\oplus 1)$. Using the preceding lemma, we get that $\Cat^{\tcol=1}$ corresponds to the quantum group $\tilde H$.
\end{proof}

\begin{prop}
\label{P.prods2}
Let $\Cat$ be a category of partitions with extra singletons. Denote by $G$ the quantum group corresponding to $\Cat$ and by $H$ the quantum group corresponding to $\Cat^{\lincol}$.
\begin{enumerate}
\item If $\singext\in\Cat$, then $\Cat=\langle\Cat^{\lincol},\singext\rangle^{\tcol}$ and $G=H\times E$.
\item If $\singext\not\in\Cat$, $\singleton\otimes\singext\not\in\Cat$, but $\positionerext\in\Cat$, then $\Cat=\langle\Cat^{\lincol},\positionerext\rangle^{\tcol}$ and $G=H\times\hat\Z_2$.
\end{enumerate}
\end{prop}
\begin{proof}
Suppose $\singext\in\Cat$ and take any $p\in\Cat$. Then since $\singext\in\Cat$, we can remove all extra singletons in $p$ by composition and obtain some $q\in\Cat^{\lincol}$. The partition $p$ can be obtained by reversing this process, i.e. taking $q\in\Cat^{\lincol}$ and tensoring it with extra singletons $\singext$, which proves that $\Cat$ is generated by $\Cat^{\lincol}$ and $\singext$.

Similarly for the second case. The conditions $\singext\not\in\Cat$ and $\singleton\otimes\singext\not\in\Cat$ are equivalent to assuming that any partition $p\in\Cat$ contains an even number of extra singletons. This allows to reconstruct any partition $p\in\Cat$ from some non-colored version $q\in\Cat^{\lincol}$ using $\positionerext$ (using composition with partitions of the form $\idpart\cdots\idpart\positionerext\idpart\cdots\idpart$ we can move any extra singleton to any position).

The quantum group picture follows from Proposition \ref{P.prods}.
\end{proof}

\begin{prop}
\label{P.intgroup}
Let $\Cat$ be a category of partitions with extra singletons. Denote by $G$ the quantum group corresponding to $\Cat$, by $H$ the quantum group corresponding to $\Cat^{\lincol}$ and by $\tilde H$ the quantum group corresponding to $\Cat^{\tcol=1}$. Then
$$\tilde H\times E\subset G\subset H*\hat\Z_2.$$
\end{prop}
\begin{proof}
According to Lemma \ref{L.H} and Proposition \ref{P.prods}, the quantum group $H*\hat\Z_2$ corresponds to the category $\langle\Cat^{\lincol}\rangle^{\tcol}$. According to Lemma \ref{L.tilH} and Proposition \ref{P.prods}, the quantum group $H\times E$ corresponds to the category $\langle\Cat^{\tcol=1},\tcol\rangle^{\tcol}=\langle\Cat,\tcol\rangle^{\tcol}$. We indeed have
\[\langle\Cat,\tcol\rangle^{\tcol}\supset\Cat\supset\langle\Cat^{\lincol}\rangle^{\tcol}.\qedhere\]
\end{proof}

\begin{ex}
Consider the category $\Cat:=\langle\Psa\rangle^{\tcol}$. It holds that $\Cat^{\lincol}=\langle\singleton\otimes\singleton\rangle$. Indeed, one can easily see that $\singleton\otimes\singleton$ is generated by $\Psa$ (compose $(\Psa\otimes\Psa)\cdot(\singext\otimes\singext)$), which proves the inclusion $\supset$. Conversely, one can see that all partitions in $\Cat$ have blocks of size at most two and we can also prove that $\positionerpart\not\in\Cat$ (otherwise we would have $\positionerext\in\Cat$ and hence $\globcolext\in\Cat$, which is not the case according to the classification in Proposition \ref{P.ncclass}). Hence, we have the inclusion $\subset$. Similarly, one can prove that $\Cat^{\tcol=1}=\langle\singleton\rangle$.

The category $\langle\singleton\otimes\singleton\rangle$ corresponds to the quantum group $B_N^{\# +}$, which is a subgroup of $O_N^+$ given by the relation $s:=\sum_kv_{ik}=\sum_kv_{kj}$ for all $i,j=1,\dots,N$ (originally defined in \cite{Web13}). The category $\langle\singleton\rangle$ corresponds to the quantum group $B_N^+$, which is a quantum subgroup of $B_N^{\# +}$ given by $s=1$ (originally defined in \cite{BS09}). According to Proposition \ref{P.intgroup} we have that $\Cat$ corresponds to a quantum group $G$ with
$$B_N^+\times E\subset G\subset B_N^{\# +}*\hat\Z_2.$$
In fact, as a subgroup of $B_N^{\#+}*\hat\Z_2$, it is given by the relation $r=s$ arising from the partition \Psa.

Note that $G$ is in fact, as a quantum group, isomorphic to $B_N^{\#+}$ (just take the $*$-isomorphism $C(G)\to C(B_N^{\#+})$ mapping $v_{ij}\mapsto v_{ij}$ and $r\mapsto s$). Nevertheless, $B_N^+$ is the maximal compact matrix quantum group $H$ that can be embedded in $G$ in the form $H\times E\subset G$ (that is, having a surjective $*$-homomorphism $C(G)\to C(H)$ mapping $r\mapsto 1$).

For those considerations, note the important distinction between two quantum groups being isomorphic (existence of a C*-algebra isomorphism that preserves the comultiplication) and being equal as {\em matrix} quantum groups (the fundamental representations must coincide as well; in particular, they must have the same size).
\end{ex}

\section{Classification of categories with extra singletons}
\label{sec.classification}

We are not going to solve the classification problem for categories of partitions with extra singletons explicitly. In the following subsection, we are going to treat some special cases. Then we are going to transform the rest to another problem for which we already have partial classification results.

\begin{rem}
\label{R.Fre}
As we already mentioned, our classification problem is closely related to the classification of partition categories with two self-dual colors, which was solved in the non-crossing case by Freslon in \cite{Fre19}. Let us state here explicitly the relation to our work. Strictly speaking we are solving two different problems as Freslon looks for quantum groups $G'$ with $S_N\subset G'\subset O_N^+ * O_N^+$ while we are looking for quantum groups $G$ with $S_N\times E\subset G \subset O_N^+ *\Z_2$. Nevertheless, any category of partitions with extra singletons can be also considered as a category with two self-dual colors. Hence, the classification of non-crossing categories of extra-singletons (summarized in Section \ref{secc.noncross}) was already intrinsically contained in \cite{Fre19} as well as many quantum group relations that are discussed in Section \ref{sec.prods}. On the other hand, in our work, we do not restrict to the non-crossing case and even here we state the results in much more explicit way (see Tables \ref{tab.globcol}, \ref{tab.nc}) than in \cite{Fre19}.

Also, note that the classification of categories with extra singletons is only a side aspect of our present article. The new products, as defined in our article, or results such as our Theorems \ref{T.A}, \ref{T.B}, and \ref{T.C} have not been considered in \cite{Fre19}.
\end{rem}

\subsection{Partitions of odd length}
\label{secc.odd}

We first show that the case of partitions of odd length can be reduced to the case of partitions of even length.

\begin{lem}
Let $\Cat$ be a category of partitions with extra singletons. Suppose $\Cat$ contains a partition of odd length. Then $\singext\in\Cat$ or $\singleton\in\Cat$.
\end{lem}
\begin{proof}
Suppose $p\in\Cat$ has odd length $l>1$. Without loss of generality, suppose that $p$ has lower points only, i.e. $p\in\Cat(0,l)$. Then there must be two neighboring points in $p$ (alternatively the first and the last point) of the same color, so they can be contracted. By induction, we can contract any partition of odd length to a partition of length one, i.e. a singleton or an extra singleton.
\end{proof}

For the case $\singext\in\Cat$, recall Proposition \ref{P.prods2} saying that the category $\Cat$ is determined by the one-colored category $\Cat^{\lincol}$ and the corresponding quantum group is of the form $G=H\times E$. Thus, the classification of such categories reduces to the one-colored case which is done \cite{RW16}.

The case, when the singleton $\singleton$ is contained in the category can be transformed to the case when it is not.

\begin{lem}
\label{L.singleton}
Let $\Cat$ be a category of partitions with extra singletons such that $\singleton\in\Cat$. Then $\Cat=\langle\tilde\Cat,\singleton\rangle^{\tcol}$, where
$$\tilde\Cat=\{p\in\Cat\mid\text{$|p|$ is even}\}.$$
\end{lem}
\begin{proof}
The inclusion $\supset$ is obvious. For the converse, consider $p\in\Cat$ with odd length. Then we have $p\otimes\singleton\in\tilde\Cat$, so $p\in\langle p\otimes\singleton,\singleton\rangle^{\tcol}\subset\langle\tilde\Cat,\singleton\rangle^{\tcol}$.
\end{proof}

\subsection{Connection between extra-singleton categories and two-colored categories}

\begin{defn}
\label{D.F}
We define a functor $F\colon\Part^{\tcol}\to\Part\twocol$ as follows.
\begin{itemize}
\item Consider an object in $\Part^{\tcol}$, that is, a word $w$ over $\Alphabet$. Then $F(w)$ is obtained by coloring all the points in $w$ with alternating white and black color starting with white and then deleting all extra singletons. In particular, two neighboring points in $F(w)$ have the same color if and only if the corresponding points in $w$ are separated by an odd number of $\tcol$.
\item Consider a partition $p\in\Part^{\tcol}(w_1,w_2)$. Then $F(p)$ is a two-colored partition with upper color pattern $F(w_1)$ and lower color pattern $F(w_2)$ with the same block structure as $p$ (ignoring the extra singletons).
\end{itemize}
\end{defn}

\begin{ex}
\label{ex.F}
As a typical example, take the partition from Equation \eqref{eq.pext}. We map it as follows
$$\BigPartition{
\Psingletons 0to0.25:5
\Psingletons 1to0.75:1
\Ppoint0 \Lt:1,4
\Ppoint1 \Ut:3,4
\Pblock 0to0.3:2,3
\Pblock 1to0.7:2,5
\Pline(2.5,0.3)(2.5,0.7)
}\mapsto
\BigPartition{
\Psingletons 0to0.25:5
\Psingletons 1to0.75:1
\Pblock 0to0.3:2,3
\Pblock 1to0.7:2,5
\Pline(2.5,0.3)(2.5,0.7)
\Ppoint1 \Pw:1,5
\Ppoint1 \Pb:2
\Ppoint0 \Pw:3,5
\Ppoint0 \Pb:2
}.$$
That is, we color the odd points (i.e.\ first, third and fifth on both rows) with white color and the even points (the second and fourth on both rows) with black. Then we erase all the triangles. Further examples are the following
$$
\idpart\mapsto\idpart[w/w],\quad \idext\otimes\idpart\mapsto\idpart[b/b],\quad \positionerext\mapsto\idpart[b/w],\quad\globcolext\mapsto\idpart[b/w]\otimes\idpart[w/b],
$$
Note, in particular, that we have $F(\idext\otimes p)=\overline{F(p)}$, where the bar denotes the color inversion $\wcol\leftrightarrow\bcol$.
Note also that the image of a partition $p\in\Part^{\tcol}$ is invariant with respect to adding a pair of consecutive extra singletons to $p$ and adding arbitrary amount of extra singletons to the end of the upper or lower row. Conversely, for any $\tilde p\in\Part\twocol$, its preimages differ only by such changes. In particular, any word $w$ over $\{\wcol,\bcol\}$ and any partition $\tilde p\in\Part\twocol$ has a unique shortest preimage.
\end{ex}

\begin{rem}
\label{R.FTp}
Since the partition structure is not changed by the functor $F$, it follows that the maps $T_p$ and $T_{F(p)}$ for a given $p\in\Part^{\tcol}(w_1,w_2)$ are exactly the same maps $(\C^N)^{\otimes k'}\to(\C^N)^{\otimes l'}$, where $k'$ and $l'$ are the lengths of the words $F(w_1)$ and $F(w_2)$. The only thing that changes is the interpretation of those maps. The map $T_p$ is considered as an intertwiner in $\Mor(u^{\otimes w_1},u^{\otimes w_2})$ with $u^{\lincol}=v$ and $u^{\tcol}=r$ for some quantum group $G=(C(G),v\oplus r)\subset O_N^+*\hat\Z_2$. In contrast, the map $T_{F(p)}$ is interpreted as an intertwiner in $\Mor(\tilde v^{\otimes F(w_1)},\tilde v^{\otimes F(w_2)})$, where $\tilde v^{\wcol}=\tilde v$ and $\tilde v^{\bcol}=\bar{\tilde v}$, for some quantum group $\tilde G=(C(\tilde G),\tilde v)\subset U_N^+$.
\end{rem}

\begin{ex}
\label{ex.FTp}
As an example, let us again take the partition $p=
\BigPartition{
\Psingletons 0to0.25:5
\Psingletons 1to0.75:1
\Ppoint0 \Lt:1,4
\Ppoint1 \Ut:3,4
\Pblock 0to0.3:2,3
\Pblock 1to0.7:2,5
\Pline(2.5,0.3)(2.5,0.7)
}$. We associate to it a map $T_p\colon \C^N\otimes\C^N\otimes\C^N\to\C^N\otimes\C^N\otimes\C^N$ by Equation \eqref{eq.Tp} ignoring the extra singletons. That is, we have
$$T_p(e_{i_1}\otimes e_{i_2}\otimes e_{i_3})=\delta_{i_2i_3}\,e_{i_2}\otimes e_{i_2}\otimes\sum_{k=1}^N e_k.$$
As we just mentioned, it coincides with the $T_p$ map associated to the partition
\Partition{
\Psingletons 0to0.3:3
\Psingletons 1to0.7:1
\Pblock 0to0.3:1,2
\Pblock 1to0.7:2,3
\Pline(1.5,0.3)(2.5,0.7)
} and hence also with the map associated to $F(p)=
\Partition{
\Psingletons 0to0.3:3
\Psingletons 1to0.7:1
\Pblock 0to0.3:1,2
\Pblock 1to0.7:2,3
\Pline(1.5,0.3)(2.5,0.7)
\Ppoint0 \Pw:2,3
\Ppoint0 \Pb:1
\Ppoint1 \Pw:1,3
\Ppoint1 \Pb:2
}$. The meaning of the partition $p$ is the relation
$$T_p(v\otimes v\otimes r\otimes r\otimes v)=(r\otimes v\otimes v\otimes r\otimes v)T_p,$$
which can also be written as
$$\delta_{i_1i_2}\sum_{l=1}^Nv_{lj_1}v_{i_1j_2}r^2v_{i_1j_3}=\delta_{j_2j_3}\sum_{k=1}^Nrv_{i_1j_2}v_{i_2j_2}rv_{i_3k}.$$
The meaning of the partition $F(p)$ is the relation
$$T_p(\tilde v\otimes \bar{\tilde v}\otimes \tilde v)=(\bar{\tilde v}\otimes \tilde v\otimes \tilde v)T_p,$$
which can also be written as
$$\delta_{i_1i_2}\sum_{l=1}^N\tilde v_{lj_1}\tilde v_{i_1j_2}^*\tilde v_{i_1j_3}=\delta_{j_2j_3}\sum_{k=1}^N\tilde v_{i_1j_2}^*\tilde v_{i_2j_2}\tilde v_{i_3k}.$$
We can say that the functor $F$ changes the corresponding relations by mapping $v_{ij}r\mapsto \tilde v_{ij}$ and $rv_{ij}\mapsto\tilde v_{ij}^*$. See also Proposition \ref{P.FG}.
\end{ex}

\begin{prop}
\label{P.F}
The map $F$ satisfies:
\begin{enumerate}
\item $F(w_1\otimes w_2)=F(w_1)\otimes F(w_2)$ or $F(w_1\otimes w_2)=F(w_1)\otimes\overline{F(w_2)}$.
\item For $p,q$ of even length, we have $F(p\otimes q)=F(p)\otimes F(q)$ or $F(p\otimes q)=F(p)\otimes\overline{F(q)}$.
\item If $p$ and $q$ are composable, then $F(p)$ and $F(q)$ are composable and $F(qp)=F(q)F(p)$,
\item $F(p^*)=F(p)^*$.
\end{enumerate}
Thus, $F$ is a unitary functor (by (3) and (4)), but not a monoidal functor (by (1) and (2)).
\end{prop}
\begin{proof}
The proof is straightforward. Note that in (1) we apply the color inversion if and only if the length of $w_1$ is odd. In (2) we use the fact that, for $p\in\Part^{\tcol}(k,l)$ of even length, we have that either both $k$ and $l$ are even and we do not have to apply the color inversion for $F(q)$ or both $k$ and $l$ are odd and then we apply the color inversion for both the upper and the lower row of $F(q)$.
\end{proof}

\begin{defn}
We denote by $\Part^{\tcol}\even$ the set of all partitions with extra singletons having even length. From now on we will consider $F$ to be defined only on $\Part^{\tcol}\even$. In particular, given a category $\tilde\Cat\subset\Part\twocol$, we denote by $F^{-1}(\tilde\Cat)$ its preimage inside $\Part^{\tcol}\even$.
\end{defn}

\begin{thm} The map $F$ defines the following one-to-one correspondence.
\label{T.F}
\begin{enumerate}
\item Let $\Cat$ be a category of partitions with extra singletons of even length. Then $F(\Cat)$ is a category of two-colored partitions, which is invariant with respect to color inversions.
\item Let $\tilde\Cat$ be a category of two-colored partitions invariant with respect to color inversions. Then $F^{-1}(\tilde\Cat)$ is a category of partitions with extra singletons of even length.
\end{enumerate}
It holds that $F^{-1}(F(\Cat))=\Cat$ and $F(F^{-1}(\tilde\Cat))=\tilde\Cat$.
\end{thm}
\begin{proof}
Consider a category $\Cat\subset\Part\even^{\tcol}$. As mentioned in Example \ref{ex.F}, we have that $F(\idext\otimes p)$ is the color inversion of $F(p)$, so $F(\Cat)$ is indeed closed under color inversions. From Proposition \ref{P.F} it directly follows that $F(\Cat)$ is closed under involution. It is also closed under tensor products since we have that either $F(p)\otimes F(q)=F(p\otimes q)$ or $F(p)\otimes F(q)=F(p\otimes\idext\otimes q)$. To check that $F(\Cat)$ is closed under compositions, it is enough to prove that for any composable pair $\tilde p,\tilde q\in F(\Cat)$ there exist $p,q\in\Cat$ composable such that $\tilde p=F(p)$ and $\tilde q=F(q)$. It suffices to take $p$ with the shortest possible lower row (with no extra singletons at the end and neighboring extra singletons anywhere) and $q$ with the shortest possible upper row.

The part (2) is proven similarly.

The equality $F(F^{-1}(\tilde\Cat))=\tilde\Cat$ is surely satisfied since it holds for any map $F$.

Since $\singext\otimes\singext\in\Cat$ for any category $\Cat$, we have that any category is closed under adding or removing pairs of neighboring extra singletons. Since also $\idext\in\Cat$, we can also add arbitrary amount of extra singletons to the end of lower and upper row. Consequently, any category $\Cat$ contains with any element $p\in\Cat$ the whole preimage $F^{-1}(F(p))$. Therefore, we also have $F^{-1}(F(\Cat))=\Cat$. This also proves that the described relationship is indeed a one-to-one correspondence.
\end{proof}

\begin{prop}
\label{P.Fgen}
Let $S\subset\Part^{\tcol}$ be a set of partitions with extra singletons. Then $F(\langle S\rangle^{\tcol})=\langle F(S)\rangle$.
\end{prop}
\begin{proof}
The assertion follows from Theorem \ref{T.F}, namely from the fact that both $F$ and $F^{-1}$ map a category to a category. We surely have the inclusion $\supset$ since obviously $F(S)\subset F(\langle S\rangle^{\tcol})$ and $F(\langle S\rangle^{\tcol})$ is a category, so it must contain the category generated by $F(S)$. For the converse inclusion, we surely have $S\subset F^{-1}(\langle F(S)\rangle)$. Since we have a category on the right-hand side, it must contain $\langle S\rangle^{\tcol}$ and then we just apply $F$ to both sides.
\end{proof}

\begin{defn}
Consider a quantum group $G\subset O_N^+*\hat\Z_2$ with fundamental representation $v\oplus r$. Denote $\tilde v_{ij}:=v_{ij}r$ and let $A$ be the C*-subalgebra of $C(G)$ generated by $\tilde v_{ij}$. Then $\tilde G:=(A,u)$ is called the \emph{glued version} of $G$.
\end{defn}

\begin{rem}
It is easy to check that the comultiplication on $G$ satisfies $\Delta(\tilde v_{ij})=\sum_k\tilde v_{ik}\otimes\tilde v_{kj}$, so its restriction provides a comultiplication on $\tilde G$. Thus, $\tilde G$ is a compact matrix quantum group.
\end{rem}

\begin{rem}
The definition generalizes the glued product construction from Subsection \ref{secc.glued}. It is easy to see that $G\tilstar\hat\Z_2$ is the glued version of $G*\hat\Z_2$ and $G\tiltimes\hat\Z_2$ is the glued version of $G\times\hat\Z_2$.
\end{rem}

\begin{prop}
\label{P.FG}
Consider a category $\Cat\subset\Part^{\tcol}\even$. Denote by $G=(C(G),v\oplus r)$ the quantum group corresponding to $\Cat$ and by $\tilde G=(C(\tilde G),\tilde v)$ the quantum group corresponding to the category $\tilde\Cat:=F(\Cat)$. Then there is an injective $*$-homomorphism $\iota\colon C(\tilde G)\to C(G)$ mapping $\tilde v_{ij}\mapsto v_{ij}r$. In other words, $\tilde G$ is the glued version of $G$.
\end{prop}
\begin{proof}

To prove the existence of a $*$-homomorphism $\iota\colon C(\tilde G)\to C(G)$ mapping $\tilde v_{ij}\mapsto v_{ij}r$, we need to show that the elements $\tilde v_{ij}':=v_{ij}r\in C(G)$ satisfy all the relations of the generators $\tilde v_{ij}\in C(\tilde G)$. We essentially did this already in Remark~\ref{R.FTp}. Indeed, all relations in $C(\tilde G)$ are of the form $T_{\tilde p}\tilde v^{\otimes \tilde w_1}=\tilde v^{\otimes\tilde w_2}T_{\tilde p}$ for some $\tilde p\in\tilde\Cat(\tilde w_1,\tilde w_2)$. Take any preimage $p\in\Cat(w_1,w_2)$, $p\in F^{-1}(\tilde p)$. We showed in Remark~\ref{R.FTp} that $T_p=T_{\tilde p}$. One can also check that $\tilde v'^{\otimes \tilde w_1}=u^{\otimes w_1}$ and $\tilde v'^{\otimes\tilde w_2}=u^{\otimes w_2}$ (as usual, we take $u=v\oplus r$). So, we have
$$T_{\tilde p}\tilde v'^{\otimes\tilde w_1}=T_p u^{\otimes w_1}=u^{\otimes w_2}T_p=\tilde v'^{\otimes \tilde w_2}T_{\tilde p}.$$

To prove the injectivity, we are going to use a similar trick as in \cite{TW17}. We will show that there is a $*$-homomorphism $\beta\colon C(G)\to M_2(C(\tilde G))$ mapping
$$r\mapsto r':=\begin{pmatrix}0&1\\1&0\end{pmatrix},\quad v_{ij}\mapsto v_{ij}':=\begin{pmatrix}0&\tilde v_{ij}\\\tilde v_{ij}^*&0\end{pmatrix}.$$
If we prove that such a homomorphism exists, then it is easy to check that
$$(\beta\circ\iota)\tilde v_{ij}=\tilde v_{ij}'':=\begin{pmatrix}\tilde v_{ij}&0\\0&\tilde v_{ij}^*\end{pmatrix},$$
so $\beta\circ\iota$ is obviously injective, which implies the injectivity of $\iota$.

The proof of existence of such a homomorphism $\beta$ is similar to the proof of existence of $\iota$. We have to prove that the elements $r'$ and $v_{ij}'$ satisfy the same relations as the generators $r$ and $v_{ij}$. Again, we have that all the relations for $r$ and $v_{ij}$ are of the form $T_pu^{\otimes w_1}=u^{\otimes w_2}T_p$ for $p\in\Cat(w_1,w_2)$. Since we assume $\Cat\subset\Part^{\tcol}\even$, we have that $p$ is of even length. Without loss of generality, we can assume that both $w_1$ and $w_2$ have even length (otherwise, consider $p\otimes\idext$, which induces obviously an equivalent relation). Any monomial in $v_{ij}'$'s and $r'$ of even length can be expressed in terms of $\tilde v_{ij}''$'s and $\tilde v_{ij}''^*$'s. Indeed, notice that $v_{ij}'r'=\tilde v_{ij}''$, so $r'v_{ij}'=(v_{ij}'r')^*=\tilde v_{ij}''^*$ and $v_{ij}'v_{kl}'=v_{ij}'r'r'v_{kl}'=\tilde v_{ij}''\tilde v_{kl}''^*$. Consequently, one can see that $u'^{\otimes w_1}=\tilde v''^{\otimes F(w_1)}$ and also $u'^{\otimes w_2}=\tilde v''^{\otimes F(w_2)}$ (denoting $u'=v'\oplus r'$). Thus, using also the equality $T_p=T_{F(p)}=T_{\overline{F(p)}}$, we have
\[T_pu'^{\otimes w_1}=T_p\tilde v_{ij}''^{\otimes F(w_1)}=\tilde v_{ij}''^{\otimes F(w_2)}T_p=u'^{\otimes w_2}T_p.\qedhere\]
\end{proof}

\begin{ex}
\label{E.Un}
In \cite{Ban97}, it was proven that $U_N^+=O_N^+\tilstar\hat\Z$. In \cite[Proposition 6.20]{TW17}, it was proven that we can actually exchange $\Z$ for $\Z_2$, so we have $U_N^+=O_N^+\tilstar\hat\Z_2$. The latter is a simple consequence of Proposition \ref{P.FG}. Indeed, the quantum group $O_N^+*\Z_2$ corresponds to the smallest category with extra singletons $\Cat:=\langle\rangle^{\tcol}$. The quantum group $U_N^+$ corresponds to the smallest two-colored category $\tilde\Cat:=\langle\rangle\twocol$, which is the image of $\Cat$ under $F$. So, $U_N^+$ is a glued version of $O_N^+*\hat\Z_2$.
\end{ex}

\subsection{An application to the theory of two-colored partitions}
This correspondence not only brings classification results for categories of partitions with extra singletons, but also conversely it brings new insight to the theory of two-colored unitary partitions.

Recall the forgetful functor $\Psi\colon\Part\twocol\to\Part$ acting on two-colored partitions by forgetting the color patterns \cite{TW18}.

\begin{lem}
\label{L.FPsi}
Let $\Cat\subset\Part$ be a category of partitions such that $\singleton\not\in\Cat$. Then
$$F(\langle\Cat,\positionerext\rangle^{\tcol})=\Psi^{-1}(\Cat).$$
\end{lem}
\begin{proof}
The left-hand side equals to $\langle F(\Cat),\idpart[b/w]\rangle^{\tcol}$ by Proposition \ref{P.Fgen}. The image $F(\Cat)$ contains some colorization of partitions in $\Cat$. Thanks to the partition $\idpart[b/w]$, the category $\langle F(\Cat),\idpart[b/w]\rangle^{\tcol}$ actually contains all the colorizations of all partitions in $\Cat$ and therefore equals to $\Psi^{-1}(\Cat)$ (see \cite[Proposition 1.4]{TW18}). Note that $F(\Cat)$ is defined only if $\Cat$ contains only partitions of even length, which is equivalent to the assumption $\singleton\not\in\Cat$.
\end{proof}

\begin{defn}
We say that a two-colored partition $p\in\Part\twocol$ has an \emph{alternating coloring} if the color pattern of both upper and lower points alternates (between white and black), the color of the first points of both rows coincide, and the color of the last point of both rows coincide (consequently, $p$ is of even length). For a two-colored category $\Cat\subset\Part\twocol$, we denote by $\Alt\Cat$ the category generated by elements of $\Cat$ that have an alternating coloring. For an ordinary category $\Cat\subset\Part$ we denote $\Alt\Cat:=\Alt\Psi^{-1}(\Cat)$ the category generated by alternating colored partitions in $\Cat$.
\end{defn}

\begin{lem}
\label{L.FAlt}
Let $\Cat\subset\Part$ be a category such that $\singleton\not\in\Cat$. Then
$$F(\langle\Cat\rangle^{\tcol})=\Alt\Cat.$$
\end{lem}
\begin{proof}
Follows directly from the definition of $F$ and $\Alt$.
\end{proof}

\begin{rem}
The operation $\Alt$ for two-colored categories in general corresponds to the operation $\Cat\mapsto\langle\Cat^{\lincol}\rangle^{\tcol}$ for categories with extra singletons. More precisely, we have $\Alt\Cat=F\langle(F^{-1}(\Cat))^{\lincol}\rangle^{\tcol}$.
\end{rem}

\begin{prop}
\label{P.tilstar}
Let $\Cat\subset\Part$ be a category of partitions with $\singleton\not\in\Cat$ and denote by $G\subset O_N^+$ the corresponding quantum group. Then $\Alt\Cat$ corresponds to $G\tilstar\hat\Z_2$.
\end{prop}
\begin{proof}
From Proposition \ref{P.prods}, it follows that $\langle\Cat\rangle^{\tcol}$ corresponds to the quantum group $G*\hat\Z_2$. From Lemma \ref{L.FAlt} it follows that $\Alt\Cat$ is its image under $F$. By Proposition \ref{P.FG} this implies that it corresponds to the glued version of $G*\hat\Z_2$, which is $G\tilstar\hat\Z_2$.
\end{proof}

\begin{rem}
We will study glued and tensor complexifications in a separate article in more detail. It is possible to show that, exchanging $\Z_2$ for $\Z$, the proposition still holds true and one can then actually drop the assumption $\singleton\not\in\Cat$.
\end{rem}

\section{New interpolating products}
\label{sec.prods}

\subsection{Quantum group degree of reflection}

Recall that given a quantum group $G=(C(G),u)$, we can construct a quantum subgroup of $G$ -- so-called \emph{diagonal subgroup} -- imposing the relation $u_{ij}=0$ for all $i\neq j$. If we, in addition, impose the relation $u_{ii}=u_{jj}$ for all $i$ and $j$, we get a quantum group corresponding to a C*-algebra generated by a single unitary. Therefore, it must be a dual of some cyclic group.

\begin{defn}
Let $G$ be a quantum group and denote by $\hat\Gamma$ the quantum subgroup of $G$ given by $u_{ij}=0$, $u_{ii}=u_{jj}$ for all $i\neq j$. The order of the cyclic group $\Gamma$ is called the \emph{degree of reflection} of $G$. If the order is infinite, we set the degree of reflection to zero.
\end{defn}

Denoting by $k$ the degree of reflection of a quantum group $G$, the definition says that there is a $*$-homomorphism $\phi\colon C(G)\to\C^*(\Z_k)$ mapping $u_{ij}\mapsto \delta_{ij}z$, where $z$ is the generator of $C^*(\Z_k)$. (We put $\Z_k=\Z$ for $k=0$.) Such a homomorphism also exists for $k$ any divisor of the degree of reflection.

In \cite{TW18}, a definition of the degree of reflection was formulated in terms of the associated representation category. We recall this definition here and prove that it is actually equivalent to our definition.

\begin{defn}
Let $w$ be a word over the alphabet $\{\wcol,\bcol\}$. We define $c(w)$ to be the difference between the number of $\wcol$ and the number of $\bcol$ in $w$. For $G=(C(G),u)$ a quantum group, and $T\in\Mor(u^{\otimes w_1},u^{\otimes w_2})$ an intertwiner, we denote $c(T):=c(w_2)-c(w_1)$.
\end{defn}

\begin{prop}
\label{P.k}
Let $G=(C(G),u)$ be a quantum group and denote by $k$ its degree of reflection. Then
$$\{c(T)\mid T\neq 0,T\in\Rep G\}=k\Z.$$
\end{prop}
\begin{proof}
It is easy to see that the considered set is a subgroup of $\Z$ (this provides the categorical definition of the degree of reflection, see \cite[Lemma 2.6 and Proposition 2.7]{TW18}), so let us denote it by $\tilde k\Z$. We need to prove that $\tilde k=k$.

First, we prove that $\tilde k$ is a multiple of $k$. Take an intertwiner $T\in\Mor(u^{\otimes w_1},u^{\otimes w_2})$ with $l:=c(T)=c(w_2)-c(w_1)$, so we have $Tu^{\otimes w_1}=u^{\otimes w_2}T$. Applying the $*$-homomorphism $C(G)\to C^*(\Z_k)$, $u_{ij}\mapsto\delta_{ij}z$, we get $Tz^{c(w_1)}=z^{c(w_2)}T$, so $z^lT=T$. Consequently, if $T\neq 0$, we must have $z^l=1$, that is, $l$ is a multiple of $k$.

Now, we prove that $k$ is a multiple of $\tilde k$. To do this, it is enough to show that there is a $*$-homomorphism $C(G)\to C^*(\Z_{\tilde k})$ mapping $u_{ij}\mapsto \delta_{ij}z$. By Tannaka--Krein duality, all relations in $C(G)$ can be deduced from the relations of the form $Tu^{\otimes w_1}=u^{\otimes w_2}T$, where $T\in\Mor(u^{\otimes w_1},u^{\otimes w_2})$. Thus the desired homomorphism exists since those relations hold in $C^*(\Z_{\tilde k})$ after applying $u_{ij}\mapsto \delta_{ij}z$.
\end{proof}

\subsection{The general case}

In this subsection, we answer the question, whether there are some quantum groups interpolating the free product $G*H$ and the tensor product $G\times H$, for any given quantum groups $G$ and $H$. In the following definition, we give a very general definition of such interpolating products. Then we discuss other possibilities. In the subsequent subsections, we will then focus on the case when $G$ is orthogonal easy quantum group and $H=\hat\Z_2$.

\begin{defn}
\label{D.prods}
Let $G$ and $H$ be compact matrix quantum groups and denote by $u$ and $v$ their respective fundamental representations. We define the following quantum subgroups of $G*H$. The product $G\ttimes H$ is defined by taking the quotient of $C(G*H)$ by the relations
\begin{equation}
\label{eq.ttimes}
ab^*x=xab^*,\qquad a^*bx=xa^*b
\end{equation}
the product $G\timess H$ is defined by the relations
\begin{equation}
\label{eq.timess}
ax^*y=x^*ya,\qquad axy^*=xy^*a
\end{equation}
the product $G\times_0 H$ by the combination of the both pairs of relations and, finally, given $k\in\N$, the product $G\times_{2k}H$ is defined by the relations
\begin{equation}
\label{eq.times}
a_1x_1\cdots a_kx_k=x_1a_1\cdots x_ka_k,
\end{equation}
where $a,b,a_1,\dots,a_k\in\{u_{ij}\}$ and $x,y,x_1,\dots,x_k\in\{v_{ij}\}$. (Equivalently, we can assume $a,b,a_1,\dots,a_k\in\spanlin\{u_{ij}\}$ and $x,y,x_1,\dots,x_k\in\spanlin\{v_{ij}\}$.)
\end{defn}

\begin{thm}
\label{T.prods}
Consider quantum groups $G,H$. Then the products from Definition \ref{D.prods} are indeed well-defined quantum groups. We have the following inclusions
$$G*H\kern1ex\begin{matrix}\lower.3ex\hbox{\rotatebox{15}{$\supset$}} & G\ttimes H & \rotatebox{345}{$\supset$}\\\raise.3ex\hbox{\rotatebox{345}{$\supset$}}  & G\timess H&\rotatebox{15}{$\supset$}\end{matrix}\kern1ex G\times_0 H\supset G\times_{2k}G\supset G\times_{2l}H\supset G\times_2H= G\times H,$$
where we assume $k,l\in\N$ such that $l$ divides $k$. The last three inclusions are strict if and only if the degree of reflection of both $G$ and $H$ is different from one.
\end{thm}
\begin{proof}
It is a direct verification that in all cases the comultiplaction passes to the quotient, so the relations provide a good definition of new quantum groups.

Denote by $u$ the fundamental representation of $G$ and by $v$ the fundamental representation of $H$. Without loss of generality, we can assume that both $u$ and $v$ are unitary representations since any representation of a quantum group is similar to a unitary one. Let us use the white circle $\wcol$ as a symbol for the representation $u$, black circle $\bcol$ for $\bar u$, white square {\tiny$\square$} for $v$ and black square {\tiny$\blacksquare$} for $\bar v$. Then the relations are actually partition relations corresponding respectively to the partitions $\Pabcacb[qwb/wbq]$, $\Pabcacb[qbw/bwq]$, $\Pabcbac[Qqw/wQq]$, $\Pabcbac[qQw/wqQ]$, and $(\crosspart[qw/wq])^{\otimes k}$. 

We can use the partition calculus to show that Relations \eqref{eq.times} imply both \eqref{eq.ttimes} and \eqref{eq.timess} for any $k\in\N$. Indeed, rotating $(\crosspart[qw/wq])^{\otimes k}$, we get $(\crosspart[bQ/Qb])^{\otimes k}$. Then, using compositions with the pair partitions, one can contract the tensor product $(\crosspart[qw/wq])^{\otimes k}\otimes (\crosspart[bQ/Qb])^{\otimes k}$ to the partition $\Partition{\Pblock 0to0.5:1,4 \Pline (2,1) (2,0) \Pline (3,1) (3,0) }[qwbQ/0wb]$, which can then be rotated to $\Pabcacb[qwb/wbq]$. The other partitions can be obtained similarly. All the remaining inclusions are clear. Note that all the arguments here can be translated into direct manipulations with the relations themselves (using the unitarity relations as well). The partition calculus provides here nothing more but a shorthand for those manipulations and makes them, hopefully, more clear.

It remains to prove the statement about strictness. Denote by $m$ the degree of reflection of $G$ and by $n$ the degree of reflection of $H$. First, suppose that $m$ and $n$ are both different from one. Then it is sufficient to prove the strictness for the corresponding subgroups $\hat\Z_m$ and $\hat\Z_n$. So, we need to prove the strictness of the following inclusions
$$\hat\Z_m\times_0 \hat\Z_n\supset \hat\Z_m\times_{2k}\hat\Z_n\supset \hat\Z_m\times_{2l}\hat\Z_n,$$

Directly from the definition, we have $\hat\Z_m\times_0\hat\Z_n=\hat\Z_m*\hat\Z_n$. Indeed, the matrices $u$ and $v$ in this case have only one entry, say $a$ and $x$. The Relations \eqref{eq.ttimes} then become trivial:
$$aa^*x=x=xaa^*,\qquad a^*ax=x=xa^*a$$
and likewise the relations \eqref{eq.timess}.

For $m=n=2$, we have that $\hat\Z_2\times_{2k}\hat\Z_2$ is the dual of the dihedral group of order $4k$, so we indeed have the strictness here. For general $m$ and $n$, let us just briefly sketch the proof. From the definition, we have that $\hat\Z_m\times_{2k}\hat\Z_n$ is the dual of the finitely presented group $\langle a,b\mid a^n=1=b^m,\;(ab)^k=(ba)^k\rangle$. We need to prove that $(ab)^l\neq (ba)^l$. To do so, let us further divide the relation $(ab)^k=1$. We obtain the so-called von Dyck group $D(m,n,k)$, which has an action on a (possibly non-Euclidean) plane. From this action, we can see that $(ab)^l$ and $(ba)^l$ are indeed different for $l<k$ (unless $m=n=2$).

Now, assuming $m=1$, we are going to show that $G\ttimes H=G\times H$. Consider a $\Z$-grading on the polynomials $\C\langle x_{ij},x_{ij}^*\rangle_{i,j}$ assigning the degree one to the variables $x_{ij}$ and degree minus one to the variables $x_{ij}^*$. Then Relations \eqref{eq.ttimes} are equivalent to $f(u_{ij},u_{ij}^*)x=xf(u_{ij},u_{ij}^*)$ for any $x\in\{v_{ij}\}$ and $f$ a homogeneous polynomial of degree zero. From Proposition \ref{P.k}, we have that there exists a non-zero intertwiner $T\in\Mor(u^{\otimes w},1)$ with $c(w)=1$. This means that there is a polynomial $g$ of degree minus one such that $g(u_{ij},u_{ij}^*)=1$. Taking any $a\in\{u_{ij}\}$ and $x\in\{v_{ij}\}$, we have that $ag(u_{ij},u_{ij}^*)$ is a polynomial in $u_{ij}$'s of degree zero. Hence, we have
\[ax=a\,g(u_{ij},u_{ij}^*)\,x=xa\,g(u_{ij},u_{ij}^*)=xa.\qedhere\]
\end{proof}

\begin{rem}
For an easy quantum group $G$ corresponding to a category $\Cat\subset\Part\twocol$, we have that its degree of reflection is equal to one if and only if $\Cat$ contains the singleton $\singleton[w]$.
\end{rem}

We could continue inventing other relations coupling somehow the factors $C(G)$ and $C(H)$ using partitions. We believe however, that the above mentioned definition is the most natural. Nevertheless, as an example of a different possibility, let us define the following.

\newcommand{\bistar}[3]{\mathbin{{}^{#1}\mathord{*}^{#2}_{#3}}}
\newcommand{\bitimes}[3]{\mathbin{{}^{#1}\mathord{\times}^{#2}_{#3}}}

\begin{defn}
Let $G$ and $H$ be quantum groups. Suppose $G$ has a one-dimensional representation $s$ and $H$ has a one-dimensional representation $r$. Then we define $G\bistar{s}{r}{k}H$ to be a quantum subgroup of $G*H$ given by the relation $(sr)^k=1$.
\end{defn}

It is easy to check that this relation indeed defines a quantum subgroup. One way to see that this subgroup should not coincide with the tensor product (at least if $G$ and $H$ are ``non-trivial enough'') is to notice that the relation is non-crossing in the following sense. Consider $T_1\in\Mor(u^{\otimes k_1},s)$ and $T_2\in\Mor(u^{\otimes k_2},r)$. Then imposing the relation means adding the intertwiner $(T_1\otimes T_2)^{\otimes k}$ to $\Mor((u\oplus v)^{\otimes 2k},1)$, which is a tensor product of intertwiners acting non-trivially either just on $u$ or just on $v$ (compare with the definition of non-crossing partitions). In particular, if $G\subset B_{N_1}^{\#+}$, so we can consider $s:=\sum_k u_{ik}$, and $H\subset B_{N_2}^{\#+}$, so we can consider $r:=\sum_k v_{ik}$, then the relation $(sr)^k=1$ corresponds to $(\singleton[w]\otimes\singleton[q])^{\otimes k}$.

This particular construction and many other relations that couple some one-dimensional subrepresentations of the factors $G$ and $H$ were already described in \cite[Section 5]{Fre19}.

\subsection{The case of $G$ orthogonal and $H=\hat\Z_2$}
From now on, let us get back to the case, where the quantum group $G=(C(G),v)$ is an orthogonal quantum group and $H=\hat\Z_2$ with fundamental representation denoted by $r$.

The product $G\ttimes\hat\Z_2=G\times_0\hat\Z_2$ is the subgroup of $G*\hat\Z_2$ given by the relations
\begin{equation}
\label{eq.ttimes2}
v_{ij}v_{kl}r=rv_{ij}v_{kl},
\end{equation}
which are the relations corresponding to the partition $\globcolext$.

The quantum group $G\times_{2k}\hat\Z_2$ is the subgroup of $G*\hat\Z_2$ given by the relations
\begin{equation}
\label{eq.times2}
v_{i_1j_1}rv_{i_2j_2}r\cdots v_{i_kj_k}r=rv_{i_1j_1}rv_{i_2j_2}\cdots rv_{i_kj_k},
\end{equation}
which correspond to the partition $(\positionerext)^{\otimes k}$.

\begin{defn}
Assume $G$ has a one-dimensional representation $s$. Recall also the definition of the product $G\bistar{s}{}{k}\hat\Z_2:=G\bistar{s}{r}{k}\hat\Z_2$ given by the relation $(sr)^k=1$. We also define the product $G\bitimes{s}{}{k}\hat\Z_2$ combining the relation $(sr)^k=1$ with Relations~\eqref{eq.ttimes2}.
\end{defn}

\begin{rem}
Since both $s$ and $r$ are representations, we have that $(sr)^k$ is a representation and hence the relation $(sr)^k=1$ indeed defines a quantum subgroup. If $G$ is an easy quantum group corresponding to a category $\Cat\subset\Part$ containing the element $\singleton\otimes\singleton$, we can choose $s:=\sum_kv_{ik}=\sum_kv_{kj}$. Then the relation $(sr)^k=1$ corresponds to the partition $(\singleton\otimes\tcol)^{\otimes k}$.
\end{rem}

\begin{rem}
Again, one can compare this construction with \cite[Section 5]{Fre19}. The difference is that instead of studying quantum subgroups of $G*H$ determined by relations involving some one-dimensional subrepresentation $r$ of $H$, we set $H:=\hat\Z_2$ and work with its one-dimensional fundamental representation. As a particular example, note that the quantum group $BO_N^{+\#}\subset O_N^+*B_N^{+\#}$ from \cite[Definition 5.2]{Fre19} is essentially defined by Relations \eqref{eq.ttimes2} if we interpret $r$ as the one-dimensional representation of $B_N^{+\#}$ given by $r=\sum_kw_{ik}$ ($w$ being the fundamental representation of $B_N^{+\#}$). Hence, $O_N^+\ttimes\hat\Z_2$ is a quantum subgroup of $BO_N^{+\#}$ given by $w_{ij}=0$ unless $i=j$ and $w_{ii}=w_{jj}$. In fact, we have
$$O_N^+\ttimes\hat\Z_2\subset O_N^+\ttimes B_N^{+\#}\subset BO_N^{+\#}\subset O_N^+*B_N^{+\#}.$$
\end{rem}

\begin{prop}
\label{P.incl2}
Let $G\subset O_N^+$ be a compact matrix quantum group having a one-dimensional representation $s$. Then
$$
\begin{array}{cl}
G*\hat\Z_2\supset G\bistar{s}{}{k}\hat\Z_2\supset G\bitimes{s}{}{k}\hat\Z_2\supset G\times\hat\Z_2\qquad&\text{for any $k\in\N$,}\\
G\bistar{s}{}{k}\hat\Z_2\supset G\bistar{s}{}{l}\hat\Z_2\qquad&\text{if $l$ divides $k$,}\\
G\bitimes{s}{}{k}\hat\Z_2\supset G\bitimes{s}{}{l}\hat\Z_2\qquad&\text{if $l$ divides $k$.}
\end{array}
$$
\end{prop}
\begin{proof}
Straightforward from the definition of the products.
\end{proof}

\begin{prop}
\label{P.timeskseq}
Let $G\subset O_N^+$ and suppose that $G$ has a representation $s:=\sum_kv_{ik}=\sum_kv_{kj}$. Then $G\bitimes{s}{}{2k}\hat\Z_2=G\times_{2k}\hat\Z_2$.
\end{prop}
\begin{proof}
We need to prove that the set of relations $\left(\sum_k v_{ik}\right)^2=1$, $(sr)^{2k}=1$, and Relations \eqref{eq.ttimes2} is equivalent to the relation $\left(\sum_k v_{ij}\right)^2=1$ together with \eqref{eq.times2}. We can do this in terms of partition calculus. That is, we need to prove the following
$$\langle\singleton\otimes\singleton,\globcolext,(\singleton\otimes\singext)^{\otimes 2k}\rangle^{\tcol}=\langle\singleton\otimes\singleton,(\positionerext)^{\otimes k}\rangle^{\tcol}.$$
The two-colored version of this equality reads as
$$\langle\singleton[w]\otimes\singleton[b],\idpart[b/w]\otimes\idpart[w/b],\singleton[w]^{\otimes 2k}\rangle=\langle\singleton[w]\otimes\singleton[b],\idpart[b/w]\otimes\idpart[w/b],\idpart[b/w]^{\otimes k}\rangle.$$
This can be proven using \cite[Lemma 3.6]{Gro18}.
\end{proof}

\subsection{The easy case}
\label{secc.easy}

Recall that a two-colored category $\Cat\subset\Part\twocol$ is called \emph{globally-colorized} if it contains the partition $\idpart[b/w]\otimes\idpart[w/b]$ \cite[Definition 2.3]{TW18}, which is equivalent to saying that $\Cat$ is invariant with respect to arbitrary color permutations of partitions with lower points only.

Given a two-colored category of partitions $\Cat\subset\Part\twocol$, we denote by $\Cat_0$ the category containing partitions $p\in\Cat(w_1,w_2)$ with zero color sum, that is, $c(p):=c(w_2)-c(w_1)=0$ \cite[Definition~3.1]{Gro18} (see \cite{TW18,Gro18} for details).

\begin{lem}
\label{L.PsiAlt}
Let $\Cat\subset\Part$ be a category of partitions. Then
$$(\Psi^{-1}(\Cat))_0=\langle\Alt\Cat,\idpart[b/w]\otimes\idpart[w/b]\rangle.$$
\end{lem}
\begin{proof}
The category $\Alt\Cat$ contains some particular zero-sum colorings of partitions in $\Cat$. Adding the globally-colorizing partition $\idpart[b/w]\otimes\idpart[w/b]$ we have that $\langle\Alt\Cat,\idpart[b/w]\otimes\idpart[w/b]\rangle$ contains all the zero-sum colorings and hence the category coincides with $(\Psi^{-1}(\Cat))_0$.
\end{proof}

\begin{defn}
Let $\Cat\subset\Part$ be a category of partitions such that $\singleton\not\in\Cat$. We define the following categories of partitions with extra singletons.
$$\Cat^{\tcol}_0:=\langle\Cat,\globcolext\rangle^{\tcol},\qquad \Cat^{\tcol}_{2k}:=\langle\Cat,(\positionerext)^{\otimes k}\rangle^{\tcol}$$
for any $k\in\N$. If $\singleton\otimes\singleton\in\Cat$, then we also define
$$\Cat^{\tcol}_k:=\langle \Cat,\globcolext,(\singleton\otimes\singext)^{\otimes k}\rangle^{\tcol}.$$
\end{defn}

Recall the proof of Proposition \ref{P.timeskseq}, where we showed that, for $\singleton\otimes\singleton\in\Cat$, the two above mentioned definitions of $\Cat_{2k}^{\tcol}$ coincide.

\begin{lem}
\label{L.Fk}
Let $\Cat\subset\Part$ be a category of partitions such that $\singleton\not\in\Cat$. Then
$$\Psi^{-1}(\Cat)_0=F(\Cat^{\tcol}_0),\qquad\langle\Psi^{-1}(\Cat)_0,\idpart[w/b]^{\otimes k}\rangle=F(\Cat^{\tcol}_{2k}),\qquad\langle\Psi^{-1}(\Cat)_0,\singleton[w]^{\otimes k}\rangle=F(\Cat^{\tcol}_k),$$
where $k\in\N$ and the last equality makes sense only if $\singleton\otimes\singleton\in\Cat$.
\end{lem}
\begin{proof}
The proof is similar in all three cases using Proposition \ref{P.Fgen}, Lemma \ref{L.PsiAlt}. Take, for example, the middle one. We have
\[F(\Cat^{\tcol}_{2k})=\langle F(\Cat),F((\positionerext)^{\otimes k})\rangle=\langle\Alt\Cat,\idpart[w/b]^{\otimes k}\rangle=\langle\Psi^{-1}(\Cat)_0,\idpart[w/b]^{\otimes k}\rangle.\qedhere\]
\end{proof}

To summarize the results and constructions presented here, we formulate the following proposition.

\begin{prop}
\label{P.prodtab}
Let $\Cat\subset\Part$ be a category of partitions such that $\singleton\not\in\Cat$ corresponding to a quantum group $G\subset O_N^+$. Then Table~\ref{tab.prods} shows the quantum groups corresponding to the various categories constructed from $\Cat$. All the categories are mutually distinct (and hence also the quantum groups for large enough $N$).
\end{prop}
\begin{table}
\begin{tabular}{l|cccc}
two-colored category $\tilde\Cat$          &  $\Alt\Cat$  &  $\Psi^{-1}(\Cat)_0$  &  $\langle\Psi^{-1}(\Cat)_0,\idpart[w/b]^{\otimes k}\rangle$  &  $\Psi^{-1}(\Cat)$\\
corresp. quantum group &  $G\tilstar\hat\Z_2$  &  $G\tiltimes\hat\Z$  &  $G\tiltimes\hat\Z_{2k}$  &  $G=G\tiltimes\hat\Z_2$\\\hline
the preimage $F^{-1}(\tilde\Cat)$   &  $\langle\Cat\rangle^{\tcol}$  &  $\langle\Cat,\globcolext\rangle^{\tcol}$  &  $\langle\Cat,(\positionerext)^{\otimes k}\rangle^{\tcol}$  &  $\langle\Cat,\positionerext\rangle^{\tcol}$ \\
corresp. quantum group &  $G*\hat\Z_2$  &  $G\ttimes\hat\Z_2$  &  $G\times_{2k}\hat\Z_2$  &  $G\times\hat\Z_2$
\end{tabular}
\medskip
\caption{Categories of partitions corresponding to various glued products and their ``$\Z_2$-unglued'' versions.}
\label{tab.prods}
\end{table}
\begin{proof}
First, let us check that the first row indeed maps to the third row under $F^{-1}$. For the first column it follows from Lemma \ref{L.FAlt}. For the last column, it follows from Lemma \ref{L.FPsi}. For the rest, it follows from Lemma \ref{L.Fk}.

Now, let us check the quantum group picture. Let us start with the upper part of the table. The first column was proven in Proposition \ref{P.tilstar}. The rest follows from \cite[Theorem 5.1]{Gro18} (see also Section 5.3 of \cite{Gro18}). For the lower part of the table, the first and last column follow from Proposition \ref{P.prods} and the rest follows directly from the definitions of the products.

The mutual inequality of the categories in the last three columns follows from \cite[Lemma 3.7]{Gro18}. Thanks to the obvious inclusions, it remains only to prove inequality between the first two columns. It can be seen that $\langle\Cat\rangle^{\tcol}$ contains only those partitions with extra singletons where we can find a pairing of the extra singletons that does not cross the blocks of color $\lincol$. Since $\globcolext$ does not satisfy this property, we have $\langle\Cat\rangle^{\tcol}\subsetneq\langle\Cat,\globcolext\rangle^{\tcol}$.
\end{proof}


\section{Concrete classification results}
\label{sec.ex}

In this section, we use Theorem \ref{T.F} to transfer the available classification results for unitary two-colored partitions to the case of categories of partitions with extra singletons.

\subsection{Globally-colorized categories}

Recall \cite[Definition 2.3]{TW18} that a category of two-colored partitions $\Cat\in\Part\twocol$ is \emph{globally-colorized} if $\pairpart[ww]\otimes\pairpart[bb]\in\Cat$ or, equivalently, $\idpart[b/w]\otimes\idpart[w/b]\in\Cat$. This holds if and only if the category $F^{-1}(\Cat)\subset\Part^{\tcol}$ contains the partition $\globcolext$ (see Example \ref{ex.F}).

All globally-colorized categories were classified in \cite{Gro18}. This result induces a classification of all categories of partitions with extra singletons containing the element $\globcolext$. The classification result can be phrased as follows.

\begin{thm}
Every category of partitions with extra singletons containing only partitions of even length is of the form $\Cat^{\tcol}_k$, where $\Cat\subset\Part$ is some category of partitions such that $\singleton\not\in\Cat$ and $k\in\N_0$ is even unless $\singleton\otimes\singleton\in\Cat$. Distinct pairs $(\Cat,k)$ define distinct categories $\Cat^{\tcol}_k$ with the exception that
\begin{align*}
\langle\singleton\otimes\singleton,(\singleton\otimes\singext)^{\otimes k}\rangle&=\langle\positionerpart,(\singleton\otimes\singext)^{\otimes k}\rangle\\
\langle\halflibpart,\singleton\otimes\singleton,\globcolext,(\singleton\otimes\singext)^{\otimes k}\rangle&=\langle\crosspart,\singleton\otimes\singleton,\globcolext,(\singleton\otimes\singext)^{\otimes k}\rangle
\end{align*}
for all $k$ odd.
\end{thm}
\begin{proof}
As already mentioned, we just apply Theorem \ref{T.F} to the classification \cite{Gro18}. The statement is then just reformulation of \cite[Theorem 3.1]{Gro18}. Using Lemma \ref{L.Fk} we find the preimage of the categories mentioned in \cite[Lemma 3.6, Lemma 3.7]{Gro18} by the functor $F$ (all of the categories are invariant with respect to the color inversions). The mutual inequality is discussed in the proof of \cite[Theorem~3.1]{Gro18}.
\end{proof}

%
%

So, if we want to obtain a list of all categories with partitions with extra singletons, we just need to take the classification of all categories of partitions $\Cat$ \cite{RW16} and construct the categories $\Cat^{\tcol}_k$. We already did similar work in the unitary two-colored case, so we can just copy the result applying the functor $F$ on \cite[Table 1]{Gro18}. The result is listed in Table~\ref{tab.globcol}. The corresponding quantum groups were described in Section \ref{sec.prods}.

\begin{table}
\begin{tabular}{lll}
Non-crossing:          & $\langle\globcolext,(\positionerext)^{\otimes k/2}\rangle$, $k\in 2\N_0$\\
                       & $\langle\fourpart,\globcolext,(\positionerext)^{\otimes k/2}\rangle$, $k\in 2\N_0$\\
                       & $\langle\singleton\otimes\singleton,\globcolext,(\singleton\otimes\singext)^{\otimes k}\rangle$, $k\in 2\N_0$\\
                       & $\langle\Pabac,\globcolext,(\singleton\otimes\singext)^{\otimes k}\rangle$, $k\in \N_0$ & \\
                       & $\langle\fourpart,\singleton\otimes\singleton,\globcolext,(\singleton\otimes\singext)^{\otimes k}\rangle$, $k\in \N_0$ & \\
$\crosspart\in\Cat$:   & $\langle\crosspart,\globcolext,(\positionerext)^{\otimes k/2}\rangle$, $k\in 2\N_0$\\
                       & $\langle\crosspart,\singleton\otimes\singleton,\globcolext,(\singleton\otimes\singext)^{\otimes k}\rangle$, $k\in\N_0$&\\
                       & $\langle\crosspart,\fourpart,\singleton\otimes\singleton,\globcolext,(\singleton\otimes\singext)^{\otimes k}\rangle$, $k\in\N_0$&\\
       & $\langle\crosspart,\fourpart,\globcolext,(\positionerext)^{\otimes k/2}\rangle$, $k\in 2\N_0$&($*$)\\
$\halflibpart\in\Cat$: & $\langle\halflibpart,\globcolext,(\positionerext)^{\otimes k/2}\rangle$, $k\in 2\N_0$&\\
                       & $\langle\halflibpart,\singleton\otimes\singleton,\globcolext,(\singleton\otimes\singext)^{\otimes k}\rangle$, $k\in 2\N_0$&\\
       & $\langle\halflibpart,\fourpart,\globcolext,(\positionerext)^{\otimes k/2}\rangle$, $k\in 2\N_0$ & $(*)$\\
                       & $\langle\halflibpart,\fourpart,h_s,\globcolext,(\positionerext)^{\otimes k/2}\rangle$, $k\in 2\N_0, s\ge 3$&($*$)\\
The rest:              & $\langle\pi_s,\globcolext,(\positionerext)^{\otimes k/2}\rangle, s\ge 2$, $k\in 2\N_0$\\
                       & $\langle\pi_l,\globcolext,(\positionerext)^{\otimes k/2}\mid l\in\N\rangle$, $k\in 2\N_0$\\
                       & $\langle A,\globcolext,(\positionerext)^{\otimes k/2}\rangle$, $A\unlhd\Z_2^\infty$, $k\in 2\N_0$&\\
\end{tabular}
\medskip
\caption{Classification of categories with extra singletons containing \protect\globcolext. Categories in the rows marked by the asterisk $(*)$ are special cases of the group-theoretical categories. See \cite{Gro18} for details on the notation.}
\label{tab.globcol}
\end{table}

%

Finally, let us mention the case when the category with extra singletons $\Cat$ contains also partitions of odd length. Recall from Subsection \ref{secc.odd} that this can happen only if $\singleton\in\Cat$ or $\singext\in\Cat$. In the latter case the classification is equivalent to the classification of ordinary categories.

\begin{prop}
Let $\Cat$ be a category of partitions with extra singletons such that $\singleton\in\Cat$, $\singext\not\in\Cat$, and $\globcolext\in\Cat$. Denote by $G$ the corresponding quantum group. Then
\begin{enumerate}
\item $\Cat$ is one of the following categories
$$\langle\singleton,\positionerext\rangle^{\tcol},\quad\langle\fourpart,\singleton,\positionerext\rangle^{\tcol},\quad\langle\singleton,\crosspart,\positionerext\rangle^{\tcol},\quad\langle\fourpart,\singleton,\crosspart,\positionerext\rangle^{\tcol}.$$
\item $G=H\times\hat\Z_2$, where $H$ corresponds to the category $\Cat^{\lincol}$, so it equals to $B_N^+$, $S_N^+$, $B_N$, or $S_N$ respectively.
\end{enumerate}
\end{prop}
\begin{proof}
According to Lemma \ref{L.singleton}, we have that $\Cat=\langle\tilde\Cat,\singleton\rangle^{\tcol}$, where $\tilde\Cat$ is some category of partitions with extra singletons with elements of even length. Thus, we get all possible categories $\Cat$ by adding the singleton $\singleton$ to the categories listed in Table~\ref{tab.globcol}. We find out that there are only the above mentioned four distinct instances. The quantum group picture then follows from Proposition \ref{P.prods}.
\end{proof}

\subsection{Non-crossing extra-singleton categories}
\label{secc.noncross}

In this section, we summarize the classification of categories of partitions with extra singletons induced by the classification of non-crossing categories of two-colored partitions obtained in \cite{TW18}. We mention only the locally colorized categories since the globally colorized ones were handled in the previous subsection.

It could be interesting to compare the results with \cite[Sections 5, 7]{Fre19}.

\begin{prop}
\label{P.ncclass}
Let $\Cat\subset\Part\twocol$ be a category of non-crossing partitions with extra singletons such that $\globcolext\not\in\Cat$. Then $\Cat$ equals to one of the categories in Table~\ref{tab.nc}.
\end{prop}
In the table, we denote
$$b_k^{\tcol}:=
\Partition{
\Pblock 0to0.8:1,3,5,9
\Ppoint 0 \Lt:2,4,6,10
\Ptext(7.5,0.3){\dots}
}
\in\Part^{\tcol}(0,2k)$$
the partition consisting of a block of length $k$ and $k$ extra singletons.
\begin{proof}
The classification of locally colorized two-colored categories was obtained in \cite[Theorem 7.2]{TW18}. All of the categories are invariant with respect to the color inversions. According to Theorem \ref{T.F}, we obtain all non-crossing categories of partitions with extra singletons of even length by applying $F^{-1}$ to this classification. Categories containing the extra singleton $\singext$ contain also the partition $\globcolext$, so they are not listed. Categories containing the singleton $\singleton$ are obtained by adding the singleton to the categories with partitions of even length.
\end{proof}

\begin{table}[ht]
\begin{align*}
&\langle\rangle^{\tcol},\quad\langle\fourpart\rangle^{\tcol},\quad\langle\fourpart,\singleton\rangle^{\tcol},\quad\langle\singleton\rangle^{\tcol},\\
&\langle b_k^{\tcol},b_d^{\tcol}\otimes b_d^{\tcol*},\LPartition{\Lt:2,5}{0.8:1,3,4,6},\fourpart\rangle^{\tcol}\qquad k,d\in\N_0\setminus\{1,2\},\;d\mid k\\
&\langle (\singleton\otimes\singext)^{\otimes k},(\singleton\otimes\singext)^{\otimes d}\otimes\idpart\otimes(\upsingext\otimes\upsingleton)^{\otimes d},\fourpart,\singleton\otimes\singleton\rangle\qquad k,d\in\N_0\setminus\{1\},\;d\mid k\\
&\langle (\singleton\otimes\singext)^{\otimes k},(\singleton\otimes\singext)^{\otimes d}\otimes\idpart\otimes(\upsingext\otimes\upsingleton)^{\otimes d},\singleton\otimes\singleton\rangle\qquad k,d\in\N_0,\;d\mid k\\
&\langle (\singleton\otimes\singext)^{\otimes k},(\singleton\otimes\singext)^{\otimes d}\otimes\idpart\otimes(\upsingext\otimes\upsingleton)^{\otimes d},\Pabac\rangle\qquad k,d\in\N_0\setminus\{1\},\;d\mid k\\
&\langle (\singleton\otimes\singext)^{\otimes k},(\singleton\otimes\singext)^{\otimes d}\otimes\idpart\otimes(\upsingext\otimes\upsingleton)^{\otimes d},(\singleton\otimes\singext)^{d/2+1}\otimes\idext\otimes\idpart\otimes(\upsingext\otimes\upsingleton)^{\otimes d/2-1}\rangle\\&\qquad k\in\N_0\setminus\{1\},\;d\in 2\N_0\setminus\{0,2\},\;d\mid k
\end{align*}
\caption{Classification of non-crossing categories with extra singletons not containing \protect\globcolext.}
\label{tab.nc}
\end{table}

\subsection{Categories of pair partitions with extra singletons}
\label{secc.pair}
Pair partitions with extra singletons are those, where all blocks of the color $\lincol$ are of size two. Those categories correspond to quantum groups $G$ such that $O_N\times E\subset G\subset O_N^+*\Z_2$. The classification of such categories can be solved by classifying unitary categories of pair partitions. For this problem, some results are already available. In \cite{MW18,MW19}, Mang and the second author classified all categories of pair partitions with neutral blocks, that is, those categories $\Cat$ such that $\langle\rangle\subset\Cat\subset\langle\crosspart[ww/ww]\rangle$. Complete classification is a work in progress. Some preliminary results are available in \cite{MW19b}.

\section{Correspondence with non-colored linear categories}
\label{sec.linear}

As was already mentioned in Section \ref{sec.prelim}, the concept of a category of partitions can be generalized by introducing a linear structure for partitions. In this section, we summarize the definition of such a structure and recall some results obtained in \cite{GW18}. Then we are going to show, how quantum groups corresponding to pair partitions with extra singletons can be described by linear combinations of non-colored partitions.

\subsection{Linear categories of partitions}
Let us fix a natural number $N\in\N$. We denote by $\Part\nlin(k,l)$ the vector space of formal linear combinations of partitions in $\Part(k,l)$. For partitions $p\in\Part(k,l)$ and $q\in\Part(l,m)$, we define in $\Part\nlin(k,m)$ their composition $qp$ as in Subsection \ref{secc.cat}, but in addition multiplied by a factor $N^{\rl(p,q)}$. Thus, for example
$$
\BigPartition{
\Psingletons 0 to 0.25:1,4
\Psingletons 1 to 0.75:1,4
\Pline (2,0) (3,1)
\Pline (3,0) (2,1)
\Pline (2.75,0.25) (4,0.25)
}
\cdot
\BigPartition{
\Pblock 0 to 0.25:2,3
\Pblock 1 to 0.75:1,2,3
\Psingletons 0 to 0.25:1,4
\Pline (2.5,0.25) (2.5,0.75)
}
=
\BigPartition{
\Pblock 0.5 to 0.75:2,3
\Pblock 1.5 to 1.25:1,2,3
\Psingletons  0.5 to  0.75:1,4
\Pline (2.5,0.75) (2.5,1.25)
\Psingletons -0.5 to -0.25:1,4
\Psingletons  0.5 to  0.25:1,4
\Pline (2,-0.5) (3,0.5)
\Pline (3,-0.5) (2,0.5)
\Pline (2.75,-0.25) (4,-0.25)
}
=N^2
\BigPartition{
\Pblock 0 to 0.25:2,3,4
\Pblock 1 to 0.75:1,2,3
\Psingletons 0 to 0.25:1
\Pline (2.5,0.25) (2.5,0.75)
}.
$$
We extend this operation bilinearly to the whole vector space. The tensor product is defined the same way as in Subsection \ref{secc.cat} and extended bilinearly. The involution is also defined the same way and extended antilinearly.

In this formalism, the mapping $T_\bullet$ assigning each partition $p\in\Part\nlin(k,l)$ a linear map $T_p\colon (\C^N)^{\otimes k}\to(\C^N)^{\otimes l}$ can also be extended linearly and then it is a monoidal unitary functor.

A collection of linear subspaces $\Kat(k,l)\subset\Part\nlin(k,l)$ containing the identity partition $\idpart$ and pair partition $\pairpart$ that is closed under the category operations is called a \emph{linear category of partitions}. Any linear category of partitions is according to the Tannaka--Krein duality assigned a compact quantum group. Moreover, in this generalization the converse also holds: For any compact quantum group $G$ such that $S_N\subset G\subset O_N^+$, its intertwiner spaces are described by some linear category of partitions $\Kat\subset\Part\nlin$.

For more details, see \cite{GW18}.

Of course it possible to introduce such a linear structure for any colored partitions. One just has to fix a number $N\in\N$ for each color representing the dimension of the corresponding representation. Then this number has to appear as a factor in the composition for each loop of this color.

In particular, in the case of categories of partitions with extra singletons, we fix a number $N\in\N$ as a dimension of the representation $v$ corresponding to the color $\lincol$. The extra singletons always correspond to a one-dimensional representation $r$, so we can ignore loops of extra singletons completely. So we can, for example, write the following.
$$
\BigPartition{
\Psingletons 0 to 0.25:4
\Psingletons 1 to 0.75:1
\Ppoint0 \Lt:1
\Ppoint1 \Ut:4
\Pline (2,0) (3,1)
\Pline (3,0) (2,1)
\Pline (2.75,0.25) (4,0.25)
}
\cdot
\BigPartition{
\Pblock 0 to 0.25:2,3
\Pblock 1 to 0.75:1,2,3
\Ppoint0 \Lt:4
\Psingletons 0 to 0.25:1
\Pline (2.5,0.25) (2.5,0.75)
}
=
\BigPartition{
\Pblock 0.5 to 0.75:2,3
\Pblock 1.5 to 1.25:1,2,3
\Psingletons  0.5 to  0.75:1
\Ppoint0.5 \Lt:4
\Pline (2.5,0.75) (2.5,1.25)
\Psingletons -0.5 to -0.25:4
\Psingletons  0.5 to  0.25:1
\Ppoint0.5 \Ut:4
\Ppoint-0.5 \Lt:1
\Pline (2,-0.5) (3,0.5)
\Pline (3,-0.5) (2,0.5)
\Pline (2.75,-0.25) (4,-0.25)
}
=N
\BigPartition{
\Pblock 0 to 0.25:2,3,4
\Pblock 1 to 0.75:1,2,3
\Ppoint0 \Lt:1
\Pline (2.5,0.25) (2.5,0.75)
}
$$
We denote by $\Part^{\tcol}\nlin(w_1,w_2)$ the vector space of formal linear combinations of partitions with extra singletons in $\Part^{\tcol}(w_1,w_2)$ and extend the category operations to those vector spaces.

\subsection{\texorpdfstring{The projections $\protect\disconnecterpart$ and $\pi$}{The projections}}
\label{secc.proj}
A linear combination of partitions $p$ is called \emph{projective} if $p^*=p$ and $p\cdot p=p$ \cite[Definition 2.7]{FW16}. Thanks to the functorial property of the mapping $T_\bullet$, we have that the projectivity of a partition implies that the corresponding linear map $T_p$ is an orthogonal projection.

Let us denote $\xi:=T_{\singleton}=\sum_{i=1}^N e_i\in\C^N$. Given a quantum group $G=(C(G),u)$ corresponding to some linear category of partitions $\Kat$, we have that the fundamental representation $u$ is reducible if and only if $\singleton\otimes\singleton\in\Kat$ \cite[Proposition 2.5(iii)]{RW15}. In this case, the subspace $\spanlin\{\xi\}$ and its orthogonal complement are the only invariant subspaces of $u$.

Note that the partition $\frac{1}{N}\disconnecterpart$ is projective and the corresponding linear map $\frac{1}{N}T_{\disconnecterpart}$ is an orthogonal projection onto $\spanlin\{\xi\}$. Consequently, we can define a projective partition $\pi_{(N)}:=\idpart-\frac{1}{N}\disconnecterpart$ \cite[Definition 5.3]{GW18}, which is assigned the map $T_{\pi_{(N)}}=1-\frac{1}{N}T_{\disconnecterpart}$, which is an orthogonal projection onto the orthogonal complement $\spanlin\{\xi\}^\perp$.

\subsection{Separating linear combinations of partitions}
\label{secc.separated}

Consider the alphabet $\Blphabet=\{\dotcol,\singleton\}$. Put $\pi^{\dotcol}:=\pi_{(N)}$, $\pi^{\singleton}:=\frac{1}{N}\disconnecterpart$ and similarly define the orthogonal projections $P^{\dotcol}:=T_{\pi^{\dotcol}}$, $P^{\singleton}:=T_{\pi^{\singleton}}$. Denote by $\Blphabet^k$ the set of all words over $\Blphabet$ of length $k$. For $w=a_1\cdots a_k\in\Blphabet^k$ we denote $\pi^{\otimes w}:=\pi^{a_1}\otimes\cdots\otimes\pi^{a_k}$ and similarly $P^{\otimes w}:=P^{a_1}\otimes\cdots\otimes P^{a_k}$.

For any $k\in\N_0$, the set of all $\pi^{\otimes w}$, $w\in\Blphabet^k$ forms a complete set of mutually orthogonal projections in the sense that
$$\pi^{\otimes w}\pi^{\otimes w}=\pi^{\otimes w},\quad (\pi^{\otimes w})^*=\pi^{\otimes w},$$
$$\pi^{\otimes w_1}\pi^{\otimes w_2}=0\quad\text{for $w_1\neq w_2$, $|w_1|=|w_2|$},$$
$$\sum_{w\in\Blphabet^k}\pi^{\otimes w}=\idpart^{\otimes k}.$$

Thus, any $p\in\Part\nlin(k,l)$ can be uniquely decomposed as
\begin{equation}
\label{eq.psum}
p=\sum_{\substack{w_1\in\Blphabet^k\\w_2\in\Blphabet^l}}p^{w_1}_{w_2},
\end{equation}
where $p^{w_1}_{w_2}=\pi^{\otimes w_2}p\pi^{\otimes w_1}$.

We will say that $p\in\Part\nlin(k,l)$ is \emph{separated} if there is $w_1\in\Blphabet^k$ and $w_2\in\Blphabet^l$ such that $p=p^{w_1}_{w_2}$. For example, for any $p\in\Part\nlin(k,l)$, all summands $p^{w_1}_{w_2}$ in the decomposition $p=\sum p^{w_1}_{w_2}$ are separated. 

%
\begin{defn}
Let $\Kat$ be a linear category of partitions such that $\singleton\otimes\singleton\in\Kat$. Take $k,l\in\N_0$, $w_1\in\Blphabet^k$ and $w_2\in\Blphabet^l$. We denote $\Kat(w_1,w_2):=\{p\in\Kat(k,l)\mid p=p^{w_1}_{w_2}\}$.
\end{defn}

\begin{lem}
\label{L.Kdecomp}
Let $\Kat$ be a linear category of partitions such that $\singleton\otimes\singleton\in\Kat$. For any $k,l\in\N_0$ we have the vector space decomposition
$$\Kat(k,l)=\bigoplus_{\substack{w_1\in\Blphabet^k\\w_2\in\Blphabet^l}}\Kat(w_1,w_2).$$
Thus, $\Kat$ is generated by separated elements.
\end{lem}
\begin{proof}
Since $\singleton\otimes\singleton\in\Kat$, we have also $\pi^{\otimes w}\in\Kat$ for any any word $w$. Hence also $p^{w_1}_{w_2}=\pi^{\otimes w_2}p\pi^{\otimes w_1}\in\Kat$ for any $p\in\Kat(k,l)$. Therefore all the summand in the decomposition \eqref{eq.psum} of any $p\in\Kat$ are contained in $\Kat$.
\end{proof}

\subsection{Basis for separated partitions}
\label{secc.dotted}
Take a partition $p\in\Part(k,l)$. Define a word $w_1\in\Blphabet^k$ in such a way that on the $i$-th position there is the letter $\singleton$ if $p$ has a singleton on the $i$-th position in the upper row. Otherwise, we put the letter $\dotcol$. Similarly we define the word $w_2\in\Blphabet^l$ corresponding to the lower row of $p$. Then we define $\dot p:=p^{w_1}_{w_2}$. We depict the linear combination $\dot p$ pictorially using the graphical representation of $p$ and replacing all the non-singleton blocks by dotted lines. The linear combinations $\dot p$ for any $p\in\Part$ are called \emph{dotted partitions}.

For example, taking $p:=\Pabcdcb$, we denote
$$\dot p:=\PDabcdcb:=(\pi^{\dotcol}\otimes\pi^{\dotcol}\otimes\pi^{\singleton})\Pabcdcb(\pi^{\singleton}\otimes\pi^{\dotcol}\otimes\pi^{\dotcol})=\Pabcdcb-\frac{1}{N}\Pabcdce-\frac{1}{N}\Pabcdeb+\frac{1}{N^2}\Pabcdef.$$

\begin{lem}
The set $\{\dot p\mid p\in\Part(k,l)\}$ forms a basis of the vector space $\Part\nlin(k,l)$ for any $k,l\in\N_0$.
\end{lem}
\begin{proof}
If we order the dotted partition with respect to the number of blocks, then the matrix of coefficients of $\dot p$ with respect to the basis of standard partitions is triangular with non-zero entries on the diagonal.
\end{proof}

Take a linear category of partitions $\Kat\subset\Part\nlin$ containing $\singleton\otimes\singleton$. From Lemma \ref{L.Kdecomp} it follows that this structure can be alternatively described as follows. As a set of objects, take all words over the alphabet $\Blphabet$. As morphism spaces take the vector spaces $\Kat(w_1,w_2)$. Those spaces can be conveniently described using the basis of dotted partitions, which are separated.

Note that from this point of view, the linear categories $\Kat$ with $\singleton\otimes\singleton$ look similar to the categories with extra singletons. However, the composition rule for dotted partitions is in general quite different from the composition rule for categories with extra singletons. Nevertheless, in Subsection \ref{secc.U} we are going to describe a category isomorphism between those structures.

\subsection{Relations corresponding to separated partitions}
The meaning of separated partitions can be seen when looking on the relations they imply. Take a linear combination $p\in\Part\nlin(k,l)$ and consider words $w_1\in\Blphabet^k$, $w_2\in\Blphabet^l$. Recall the definition of the projections $P^{\lincol}$ and $P^{\dotcol}$ from Subsection \ref{secc.proj}. Then the relation corresponding to $p^{w_1}_{w_2}$ is
\begin{align*}
P^{\otimes w_2}T_pu^{\otimes w_1}&=P^{\otimes w_2}T_pP^{\otimes w_1}u^{\otimes k}=T_{p^{w_1}_{w_2}}u^{\otimes k}\\&=u^{\otimes l}T_{p^{w_1}_{w_2}}=u^{\otimes l}P^{\otimes w_2}T_pP^{\otimes w_1}=u^{\otimes w_2}T_pP^{\otimes w_1},
\end{align*}
where $u^{\dotcol}:=P^{\dotcol}uP^{\dotcol}$ and $u^{\singleton}:=P^{\singleton}uP^{\singleton}$. Expressing this in terms of matrix elements, we can write $[u^{\dotcol}]_{ij}=u_{ij}-\frac{1}{N}r$ and $[u^{\singleton}]_{ij}=\frac{1}{N}r$. If $p$ consists only of singletons and ``through pairs'' (pair blocks with one upper and one lower point), then we have $p\pi^{\otimes w_1}=\pi^{\otimes w_2}p$, so $T_pP^{\otimes w_1}=P^{\otimes w_2}T_p$. Thus, the relation corresponding to $p^{w_1}_{w_2}$ is of the form
$$T_pu^{\otimes w_1}=u^{\otimes w_2}T_p,$$
so it has exactly the same form as the relation for $p$ except that we have to exchange the copies of $u$ by the corresponding subrepresentations.

%
%

\begin{ex}
\label{ex.halflib}
As an example, consider the half-liberating partition $p:=\halflibpart$. Its C*-algebraic relation is
\begin{equation}
\label{eq.halflib}
u_{i_1j_1}u_{i_2j_2}u_{i_3j_3}=u_{i_3j_3}u_{i_2j_2}u_{u_1j_1},
\end{equation}
which can be also written as $abc=cba$, where $a,b,c\in\spanlin\{u_{ij}\}$.

Consider the category $\Kat:=\langle\halflibpart,\singleton\otimes\singleton\rangle\nlin$. Since $\singleton\otimes\singleton\in\Kat$, we can separate the generator $p=\halflibpart$ and write
$$\Kat=\langle\singleton\otimes\singleton,p^{w_1}_{w_2}\mid w_1,w_2\in\Blphabet^3\rangle\nlin.$$

In fact, $\PDabcabd$, $\PDabcadc$, and $\PDabcdbc$ are rotations of each other and all the $p^{w_1}_{w_2}$ except for those three and $\PDabcabc$ are generated by $\singleton\otimes\singleton$. So, we can actually write
$$\Kat=\langle \PDabcabc,\PDabcabd,\singleton\otimes\singleton\rangle\nlin.$$
%

Now, according to what was written above, we get the relation for $\PDabcabc$ simply as Relation \eqref{eq.halflib}, where we replace all $u_{ij}$ by $[P^{\dotcol}uP^{\dotcol}]_{ij}=u_{ij}-\frac{1}{N}r$. So, we can write it also as $abc=cba$, where $a,b,c\in\spanlin\left\{u_{ij}-\frac{1}{N}r\right\}$. 

The relation for $\PDabcabd$ can be written as Relation \eqref{eq.halflib}, where the $u_{i_1j_1}$ (i.e. the first variable in the polynomial on the left hand side and the third variable in the polynomial on the right hand side) should be replaced by $\frac{1}{N}r$ and all the other $u_{ij}$ are replaced by $u_{ij}-\frac{1}{N}r$. So, we can write $rbc=cbr$, where $b,c\in\left\{u_{ij}-\frac{1}{N}r\right\}$.

Note that the dotted partition $\PDabcabd$ actually generates $\Pabcabc$ (the proof is essentially same as in \cite[Lemma 3.8]{Web13}), so we have $\langle \PDabcabd\rangle\nlin=\langle\halflibpart,\singleton\otimes\singleton\rangle\nlin$.
\end{ex}

\subsection{Correspondence with categories with extra singletons}
\label{secc.U}
Consider a linear category of partitions $\Kat$ such that $\singleton\otimes\singleton\in\Kat$, so it corresponds to a quantum group $G=(C(G),u)$, $S_N\subset G\subset B_N^{+\#}$, where the fundamental representation is reducible having a one-dimensional invariant subspace $\spanlin\{\xi\}$. In \cite{GW18}, an orthogonal matrix $U_{(N,\pm)}\in M_N(\C)$ was defined such that $U_{(N,\pm)}uU_{(N,\pm)}^*$ has a block structure $v\oplus r$ separating the two subrepresentations of $u$.

In \cite{GW18}, it was studied, which quantum group is generated by the $(N-1)$-dimensional subrepresentation of $u$. This can be done in two ways: either we consider the projection $P^{\dotcol}\colon\C^N\to\C^N$ onto $\spanlin\{\xi\}^{\perp}$ and study $u^{\dotcol}=P^{\dotcol}uP^{\dotcol}$ or we first apply the map $U_{(N,\pm)}$ and then project onto the subspace generated by the first $N-1$ basis vectors and study $v=V_{(N,\pm)}uV_{(N,\pm)}^*$, where $V_{(N,\pm)}$ is the coisometry formed by the first $N-1$ rows of $U_{(N,\pm)}$. To summarize, we have the following maps.

$$
\begin{tikzpicture}[baseline=(m-1-2.west)]
  \matrix (m) [matrix of math nodes,row sep=3em,column sep=4em,minimum width=2em]
  {
     \C^N & \spanlin\{\xi\}^\perp \\
     \C^N & \C^{N-1} \\};
  \path[-stealth]
    (m-1-1) edge node [left] {$U_{(N,\pm)}$} node [right] {$\sim$} (m-2-1)
    (m-1-1.east|-m-1-2) edge [->>] (m-1-2)
    (m-1-1) edge node [below] {$V_{(N,\pm)}$} (m-2-2)
    (m-2-1) edge [->>] (m-2-2)
    (m-1-2) edge node [right] {$V_{(N,\pm)}$} node [left] {$\sim$} (m-2-2);
\end{tikzpicture}
\qquad
\begin{tikzpicture}[baseline=(m-1-1.east)]
  \matrix (m) [matrix of math nodes,row sep=3em,column sep=4em,minimum width=2em]
  {
     u & u^{\dotcol} \\
     v\oplus r & v \\};
  \path[-stealth]
    (m-1-1) edge (m-2-1)
    (m-1-1) edge (m-1-2.west|-m-1-1)
    (m-1-1) edge (m-2-2)
    (m-2-1.east|-m-2-2) edge (m-2-2)
    (m-1-2) edge (m-2-2);
\end{tikzpicture}
$$

In \cite{GW18}, representation categories of $u$, $u^{\dotcol}$ and $v$ were studied using linear categories of partitions. Introducing categories with extra singletons allows us to study the representation category of $v\oplus r$. In the remainder of this article, we will define and study the category isomorphism $\U_{(N,\pm)}$ that completes the following commutative diagram.
\tikzset{
  subseteq/.style={
    draw=none,
    edge node={node [sloped, allow upside down, auto=false]{$\subseteq$}}},
  supseteq/.style={
    draw=none,
    edge node={node [sloped, allow upside down, auto=false]{$\supseteq$}}},
}
$$
\begin{tikzpicture}
  \matrix (m) [matrix of math nodes,row sep=3em,column sep=4em,minimum width=2em]
  {
     \Part\nlin & \PP_{(N)}\Part\nlin \\
     \Part\nnnlin^{\tcol} & \Part\nnnlin \\};
  \path[-stealth]
    (m-1-1) edge node [left] {$\U_{(N,\pm)}$} node [right] {$\sim$} (m-2-1)
    (m-1-1) edge [->>,bend left] node [above] {$\PP_{(N)}$} (m-1-2)
    (m-1-1.east|-m-1-2) edge [supseteq] (m-1-2)
    (m-1-1) edge [->>] node [below] {$\V_{(N,\pm)}$} (m-2-2)
    (m-2-1) edge [supseteq] (m-2-2)
    (m-1-2) edge node [right] {$\V_{(N,\pm)}$} node [left] {$\sim$} (m-2-2);
\end{tikzpicture}
$$

For the definition of $U_{(N,\pm)}$, $V_{(N,\pm)}$, $\V_{(N,\pm)}$, $\PP_{(N)}$, see \cite{GW18}.

\begin{defn}
\label{D.U}
Consider the alphabets $\Alphabet=\{\lincol,\tcol\}$ and $\Blphabet=\{\dotcol,\singleton\}$. Consider $\Part\nlin$ as a category with objects formed by words over the alphabet $\Blphabet$ as was described in Subsection \ref{secc.dotted}. We define a functor $\U_{(N,\pm)}\colon\Part\nlin\to\Part^{\tcol}\nnnlin$ as follows. On objects, $\U_{(N,\pm)}$ acts as a word isomorphism mapping $\dotcol\mapsto\lincol$ and $\singleton\mapsto\tcol$. For morphisms, we describe the action on the basis of dotted partitions. Taking a dotted partition $\dot p\in\Part\nlin(w_1,w_2)$, $\U_{(N,\pm)}$ acts blockwise. All singletons $\singleton$ are mapped to $\sqrt{N}\singext$. Any dotted block is mapped using the map $\V_{(N,\pm)}$ (see \cite[Definition 4.10]{GW18}), so
$$\U_{(N,\pm)}\dot b_l:=\V_{(N,\pm)}\dot b_l=\V_{(N,\pm)}b_l.$$
\end{defn}

\begin{prop}
\label{P.U}
$\U_{(N,\pm)}$ is indeed a monoidal unitary functor. That is, we have
\begin{enumerate}
\item $\U_{(N,\pm)}p\otimes\U_{(N,\pm)}q=\U_{(N,\pm)}(p\otimes q)$,
\item $\U_{(N,\pm)}q\cdot\U_{(N,\pm)}p=\U_{(N,\pm)}(qp)$ if $p$ and $q$ are composable,
\item $(\U_{(N,\pm)}p)^*=\U_{(N,\pm)}p^*$
\end{enumerate}
for any $p,q\in\Part\nlin$.
\end{prop}
\begin{proof}
Since $\U_{(N,\pm)}$ acts blockwise, it is clear that it behaves well with respect to the tensor product and involution. It is enough to show the functorial property (2) for dotted partitions. Here, we have to check that it behaves well in case of singletons and dotted blocks. For singletons, it is easy to see it directly. For dotted blocks it follows from \cite[Proposition 5.16]{GW18}.
\end{proof}

\begin{thm}
\label{T.U}
It holds that
$$T_{\U_{(N,\pm)}p}=U^{\otimes l}_{(N,\pm)}T_pU^{*\,\otimes k}_{(N,\pm)}$$
for any $p\in\Part\nlin(w_1,w_2)$, $w_1\in\Blphabet^k$, $w_2\in\Blphabet^l$. Thus, considering a linear category of partitions $\Kat\subset\Part\nlin$ containing $\singleton\otimes\singleton$ and the corresponding quantum group $G$, it holds that the linear category with extra singletons $\U_{(N,\pm)}\Kat$ corresponds to the quantum group $U_{(N,\pm)}GU_{(N,\pm)}^*$.
\end{thm}
\begin{proof}
Compare with \cite[Theorem 4.13]{GW18} and its proof. Again, it is enough to show the equality for block partitions. To be more precise, in this case we check it for the singleton $p=\singleton$ and for the dotted blocks $\dot b_l\in\Part\nlin(\emptyset,\dotcol^l)$.

For the singleton, we have
$$T_{\U_{(N,\pm)}\singleton}=\sqrt{N}\,T_{\singext}=\sqrt{N}\,e_N=U_{(N,\pm)}\xi=U_{(N,\pm)}T_{\singleton}.$$
For the dotted blocks it follows directly from \cite[Theorem 4.13]{GW18}.
\end{proof}

\subsection{Categories with dotted pairings}
In this subsection we present the main application of Theorem \ref{T.U}. Note that the dotted pair block is mapped by $\U_{(N,\pm)}$ to an ordinary pair block. Thus a category $\Kat$ with $\singleton\otimes\singleton\in\Kat$, where all blocks are of size at most two, i.e. $\langle\singleton\otimes\singleton\rangle\nlin\subset\Kat\subset\langle\singleton,\crosspart\rangle\nlin$ is mapped to a category $\langle\rangle^{\tcol}\subset\U_{(N,\pm)}\Kat\subset\langle\singext,\crosspart\rangle^{\tcol}$.

In this case, we can easily see, what is the inverse of $\U_{(N,\pm)}$:  we simply map all pair blocks to dotted pair blocks and all extra singletons to ordinary singletons. Since we have a partial classification of the extra singleton pair categories in the easy case, this induces a large class of new examples of non-easy linear categories of partitions corresponding to quantum groups $B_N\subset G\subset B_N^{\# +}$. We can take the classification result \cite{MW18,MW19,MW19b}, apply the functor $F$ from Section \ref{sec.classification}, obtain categories with extra singletons and apply the functor $\U$ to them.

As an example, we are going to apply the functor $\U$ to the categories $\langle\globcolext\rangle^{\tcol}$ and $\langle(\positionerext)^{\otimes k}\rangle^{\tcol}$ corresponding to the quantum groups $O_N^+\ttimes\hat\Z_2$ and $O_N^+\times_{2k}\hat\Z_2$. By this, we obtain non-easy categories corresponding to new non-easy quantum groups that are isomorphic to the original ones.


\begin{prop}
The following are non-easy and mutually distinct linear categories of partitions
$$\left\langle\Pabcdcb-\frac{1}{N}\Pabcdce-\frac{1}{N}\Pabcdeb+\frac{1}{N^2}\Pabcdef\right\rangle\nlin,$$
$$\left\langle\left(\Pabcb-\frac{1}{N}\Pabcd\right)^{\otimes k}\right\rangle\nlin,\qquad k\in\N\setminus\{1\}.$$
\end{prop}
\begin{proof}
The first generator actually equals to $\U_{(N+1,\pm)}^{-1}\globcolext=\PDabcdcb$, whereas the second one to $\U_{(N+1,\pm)}^{-1}(\positionerext)^{\otimes k}=(\PDabcb)^{\otimes k}$. Strict inclusions for the categories with extra singletons induce corresponding inclusions in our case. In particular, this proves the mutual inequality of the categories and their non-easiness. For the latter, note that the smallest easy category containing any of those above must be $\langle\Pabcb\rangle\nlin=\langle\PDabcb\rangle\nlin$.
\end{proof}

It is easy to write down the relations corresponding to those categories. Recall the relations for the partitions with extra singletons
\begin{eqnarray*}
\globcolext\kern13pt\qquad&\leftrightarrow&\qquad rv_{ij}v_{kl}=v_{ij}v_{kl}r,\\
(\positionerext)^{\otimes k}\qquad&\leftrightarrow&\qquad rv_{i_1j_1}rv_{i_2j_2}\cdots rv_{i_kj_k}=v_{i_1j_1}rv_{i_2j_2}r\cdots v_{i_kj_k}r,
\end{eqnarray*}
where $v\oplus r$ is the fundamental representation of the quantum group. The quantum groups corresponding to the above mentioned categories, i.e.\ defined by the dotted partitions $\PDabcdcb$, resp. $(\PDabcb)^{\otimes k}$ are quantum subgroups  of $B_N^{\#+}=(C(B_N^{\#+}),u)$ defined by precisely the same relations if we interpret $r$ as the one-dimensional subrepresentation $r:=\sum_ku_{ik}=\sum_k u_{kj}$ and $v:=u^{\dotcol}=u-\frac{1}{N}r$.

\bibliographystyle{halpha}
\bibliography{mybase}

\begin{thebibliography}{MRW17}
\expandafter\ifx\csname url\endcsname\relax
  \def\url#1{\texttt{#1}}\fi
\expandafter\ifx\csname doi\endcsname\relax
  \def\doi#1{\burlalt{doi:#1}{http://dx.doi.org/#1}}\fi
\expandafter\ifx\csname urlprefix\endcsname\relax\def\urlprefix{URL }\fi
\expandafter\ifx\csname href\endcsname\relax
  \def\href#1#2{#2}\fi
\expandafter\ifx\csname burlalt\endcsname\relax
  \def\burlalt#1#2{\href{#2}{#1}}\fi

\bibitem[Ban97]{Ban97}
Teodor Banica.
\newblock Le groupe quantique compact libre {$U(n)$}.
\newblock {\em Communications in Mathematical Physics}, 190(1):143--172, 1997.
\newblock \doi{10.1007/s002200050237}.

\bibitem[Bic04]{Bic04}
Julien Bichon.
\newblock Free wreath product by the quantum permutation group.
\newblock {\em Algebras and Representation Theory}, 7(4):343--362, 2004.
\newblock \doi{10.1023/B:ALGE.0000042148.97035.ca}.

\bibitem[BS09]{BS09}
Teodor Banica and Roland Speicher.
\newblock Liberation of orthogonal {Lie} groups.
\newblock {\em Advances in Mathematics}, 222(4):1461--1501, 2009.
\newblock \doi{10.1016/j.aim.2009.06.009}.

\bibitem[BV05]{BV05}
Saad Baaj and Stefaan Vaes.
\newblock Double crossed products of locally compact quantum groups.
\newblock {\em Journal of the Institute of Mathematics of Jussieu},
  4(1):135--173, 2005.
\newblock \doi{10.1017/S1474748005000034}.

\bibitem[CW16]{CW16}
Guillaume Cébron and Moritz Weber.
\newblock Quantum groups based on spatial partitions.
\newblock 2016, arXiv:\burlalt{1609.02321}{http://arxiv.org/abs/1609.02321}.

\bibitem[Fre17]{Fre17}
Amaury Freslon.
\newblock On the partition approach to schur-weyl duality and free quantum
  groups.
\newblock {\em Transformation Groups}, 22(3):707--751, 2017.
\newblock \doi{10.1007/s00031-016-9410-9}.

\bibitem[Fre19]{Fre19}
Amaury Freslon.
\newblock On two-coloured noncrossing partition quantum groups.
\newblock {\em Transactions of the Americal Mathematical Society},
  372:4471--4508, 2019.
\newblock \doi{10.1090/tran/7846}.

\bibitem[FS18]{FS18}
Amaury Freslon and Adam Skalski.
\newblock Wreath products of finite groups by quantum groups.
\newblock {\em Journal of Noncommutative Geometry}, 12(1):29--68, 2018.
\newblock \doi{10.4171/JNCG/270}.

\bibitem[FW14]{FW16}
Amaury Freslon and Moritz Weber.
\newblock On the representation theory of partition (easy) quantum groups.
\newblock {\em Journal für die reine und angewandte Mathematik},
  2016(720):155--197, 2014.
\newblock \doi{10.1515/crelle-2014-0049}.

\bibitem[Gro18]{Gro18}
Daniel Gromada.
\newblock Classification of globally colorized categories of partitions.
\newblock {\em Infinite Dimensional Analysis, Quantum Probability and Related
  Topics}, 21(04):1850029, 2018.
\newblock \doi{10.1142/S0219025718500297}.

\bibitem[GW20]{GW18}
Daniel Gromada and Moritz Weber.
\newblock Intertwiner spaces of quantum group subrepresentations.
\newblock {\em Communications in Mathematical Physics}, 376:81--115, 2020.
\newblock \doi{10.1007/s00220-019-03463-y}.

\bibitem[Maj90]{Maj87}
Shahn Majid.
\newblock Physics for algebraists: Non-commutative and non-cocommutative {H}opf
  algebras by a bicrossproduct construction.
\newblock {\em Journal of Algebra}, 130(1):17--64, 1990.
\newblock \doi{10.1016/0021-8693(90)90099-A}.

\bibitem[Maj91]{Maj89}
Shahn Majid.
\newblock Hopf-von {N}eumann algebra bicrossproducts, {K}ac algebra
  bicrossproducts, and the classical {Y}ang--{B}axter equations.
\newblock {\em Journal of Functional Analysis}, 95(2):291--319, 1991.
\newblock \doi{10.1016/0022-1236(91)90031-Y}.

\bibitem[MRW17]{MRW17}
Ralf Meyer, Sutanu Roy, and Stanis{\l}aw~Lech Woronowicz.
\newblock Semidirect products of {C}*-quantum groups: Multiplicative unitaries
  approach.
\newblock {\em Communications in Mathematical Physics}, 351(1):249--282, 2017.
\newblock \doi{10.1007/s00220-016-2727-3}.

\bibitem[MW19a]{MW19}
Alexander Mang and Moritz Weber.
\newblock Categories of two-colored pair partitions, part {II}: Categories
  indexed by semigroups.
\newblock 2019, arXiv:\burlalt{1901.03266}{http://arxiv.org/abs/1901.03266}.

\bibitem[MW19b]{MW19b}
Alexander Mang and Moritz Weber.
\newblock Non-hyperoctahedral categories of two-colored partitions, part {I}:
  New categories.
\newblock 2019, arXiv:\burlalt{1907.11417}{http://arxiv.org/abs/1907.11417}.

\bibitem[MW20]{MW18}
Alexander Mang and Moritz Weber.
\newblock Categories of two-colored pair partitions, part~{I}: Categories
  indexed by cyclic groups.
\newblock {\em The Ramanujan Journal}, 53:181--208, 2020.
\newblock \doi{10.1007/s11139-019-00149-w}.

\bibitem[NT13]{NT13}
Sergey Neshveyev and Lars Tuset.
\newblock {\em Compact Quantum Groups and Their Representation Categories}.
\newblock Société Mathématique de France, Paris, 2013.

\bibitem[RW15]{RW15}
Sven Raum and Moritz Weber.
\newblock Easy quantum groups and quantum subgroups of a semi-direct product
  quantum group.
\newblock {\em Journal of Noncommutative Geometry}, 9(4):1261--1293, 2015.
\newblock \doi{10.4171/JNCG/223}.

\bibitem[RW16]{RW16}
Sven Raum and Moritz Weber.
\newblock The full classification of orthogonal easy quantum groups.
\newblock {\em Communications in Mathematical Physics}, 341(3):751--779, 2016.
\newblock \doi{10.1007/s00220-015-2537-z}.

\bibitem[Tim08]{Tim08}
Thomas Timmermann.
\newblock {\em An Invitation to Quantum Groups and Duality}.
\newblock European Mathematical Society, Zürich, 2008.

\bibitem[TW17]{TW17}
Pierre Tarrago and Moritz Weber.
\newblock Unitary easy quantum groups: The free case and the group case.
\newblock {\em International Mathematics Research Notices},
  2017(18):5710--5750, 2017.
\newblock \doi{10.1093/imrn/rnw185}.

\bibitem[TW18]{TW18}
Pierre Tarrago and Moritz Weber.
\newblock The classification of tensor categories of two-colored noncrossing
  partitions.
\newblock {\em Journal of Combinatorial Theory, Series A}, 154:464--506, 2018.
\newblock \doi{10.1016/j.jcta.2017.09.003}.

\bibitem[VV03]{VV03}
Stefaan Vaes and Leonid Vainerman.
\newblock Extensions of locally compact quantum groups and the bicrossed
  product construction.
\newblock {\em Advances in Mathematics}, 175(1):1--101, 2003.
\newblock \doi{10.1016/S0001-8708(02)00040-3}.

\bibitem[Wan95a]{Wan95free}
Shuzhou Wang.
\newblock Free products of compact quantum groups.
\newblock {\em Communications in Mathematical Physics}, 167(3):671--692, 1995.
\newblock \doi{10.1007/BF02101540}.

\bibitem[Wan95b]{Wan95tensor}
Shuzhou Wang.
\newblock Tensor products and crossed products of compact quantum groups.
\newblock {\em Proceedings of the London Mathematical Society},
  s3-71(3):695--720, 1995.
\newblock \doi{10.1112/plms/s3-71.3.695}.

\bibitem[Wan98]{Wan98}
Shuzhou Wang.
\newblock Quantum symmetry groups of finite spaces.
\newblock {\em Communications in Mathematical Physics}, 195(1):195--211, 1998.
\newblock \doi{10.1007/s002200050385}.

\bibitem[Web13]{Web13}
Moritz Weber.
\newblock On the classification of easy quantum groups.
\newblock {\em Advances in Mathematics}, 245:500--533, 2013.
\newblock \doi{10.1016/j.aim.2013.06.019}.

\bibitem[Wor87]{Wor87}
Stanisław~L. Woronowicz.
\newblock Compact matrix pseudogroups.
\newblock {\em Communications in Mathematical Physics}, 111(4):613--665, 1987.
\newblock \doi{10.1007/BF01219077}.

\bibitem[Wor88]{Wor88}
Stanisław~L. Woronowicz.
\newblock {Tannaka-Krein} duality for compact matrix pseudogroups. {Twisted
  $SU(N)$} groups.
\newblock {\em Inventiones mathematicae}, 93(1):35--76, 1988.
\newblock \doi{10.1007/BF01393687}.

\end{thebibliography}

\end{document}